%% file: SSSurfDecomp8.tex
\documentclass[10pt]{amsart}
\theoremstyle{plain}
\usepackage{amsmath, amsfonts, amsthm, amssymb,amscd}
\usepackage{graphics}
\usepackage{color}
\usepackage{epsfig}
\usepackage[all]{xy}
\graphicspath{ {Figures/} }

\title[Kh, SFH, and Naturality]{On the naturality of the spectral sequence from Khovanov homology to Heegaard Floer homology}
\author{J. Elisenda Grigsby}
\thanks{JEG was partially supported by an NSF postdoctoral fellowship and NSF grant number DMS-0905848.}
\address{Columbia Math Dept.;2990 Broadway MC4406; NY, NY 10027}
\email{egrigsby@math.columbia.edu}

\author{Stephan Wehrli}
\thanks{SW was supported by a postdoctoral fellowship of the Fondation Sciences Math\'ematiques de Paris.}
\address{Institut de Math\'ematiques de Jussieu; Universit\'e Paris 7; 175 rue du Chevaleret; bureau 7B3; 75013 Paris, France}
\email{wehrli@math.jussieu.fr}

\theoremstyle{plain}
\newtheorem{theorem}{Theorem}[section]
\newtheorem{lemma}[theorem]{Lemma}
\newtheorem{proposition}[theorem]{Proposition}
\newtheorem{corollary}[theorem]{Corollary}

\newtheorem*{theoremA}{Theorem~\ref{thm:SurfDecomp}}
\newtheorem*{theoremAdjoinTriv}{Theorem~\ref{thm:AdjoinTrivial}}
\newtheorem*{theoremCut}{Theorem~\ref{thm:Cutting}}
\newtheorem*{theoremStack}{Theorem~\ref{thm:Stack}}

\theoremstyle{definition}
\newtheorem{notation}[theorem]{Notation}
\newtheorem{definition}[theorem]{Definition}
\newtheorem{remark}[theorem]{Remark}

\newcommand{\Ozsvath}{{Ozsv{\'a}th} }

\newcommand{\Juhasz}{{Juh{\'a}sz} }

\newcommand{\N}{\ensuremath{\mathbb{N}}}

\newcommand{\R}{\ensuremath{\mathbb{R}}}
\newcommand{\Z}{\ensuremath{\mathbb{Z}}}

\newcommand{\C}{\ensuremath{\mathbb{C}}}
\newcommand{\bL}{\ensuremath{\mathbb{L}}}

\newcommand{\PP}{\ensuremath{\mathbb{P}}}

\newcommand{\boldalpha}{\ensuremath{\mbox{\boldmath $\alpha$}}}
\newcommand{\boldbeta}{\ensuremath{\mbox{\boldmath $\beta$}}}

\newcommand{\boldeta}{\ensuremath{\mbox{\boldmath $\eta$}}}
\newcommand{\boldtheta}{\ensuremath{\mbox{\boldmath $\theta$}}}
\newcommand{\boldSigma}{\ensuremath{\mbox{\boldmath $\Sigma$}}}
\newcommand{\cP}{\ensuremath{\mathcal{P}}}
\newcommand{\cI}{\ensuremath{\mathcal{I}}}

\newcommand{\cF}{\ensuremath{\mathcal{F}}}
\newcommand{\cA}{\ensuremath{\mathcal{A}}}
\newcommand{\cM}{\ensuremath{\mathcal{M}}}
\newcommand{\Ztwo}{\ensuremath{\mathbb{Z}_2}}

\newcommand{\Torus}{\ensuremath{\mathbb{T}}}
\newcommand{\YI}{\ensuremath{Y(\cI)}}
\newcommand{\SI}{\ensuremath{S_{\cI}}}
\newcommand{\cR}{\ensuremath{\mathcal{R}}}
\newcommand{\spincs}{\ensuremath{\mathfrak{s}}}

\newcommand{\Weta}{\ensuremath{W_{\eta_0, \ldots, \eta_n}}}
\newcommand{\Wetak}{\ensuremath{W_{\eta_0, \eta_{i_1},\ldots, \eta_{i_k}}}}
\newcommand{\hatSigma}{\ensuremath{\widehat{\Sigma}}}
\newcommand{\hatP}{\ensuremath{\widehat{P}}}
\newcommand{\piSigma}{\ensuremath{\pi_\Sigma}}
\newcommand{\piP}{\ensuremath{\pi_{\mathbb{P}}}}
\newcommand{\circlesZ}{\ensuremath{(Z_1 \amalg \ldots \amalg Z_m)}}
\newcommand{\outerspinc}{\ensuremath{O_{\amalg_{\cI}}}}

\begin{document}
\bibliographystyle{plain}

\begin{abstract} In \cite{MR2141852}, Ozsv{\'a}th-Szab{\'o} established an algebraic relationship, in the form of a spectral sequence, between the reduced Khovanov homology of (the mirror of) a link $\bL \subset S^3$ and the Heegaard Floer homology of its double-branched cover.  This relationship, extended in \cite{GT07060741} and \cite{GT08071432}, was recast, in \cite{AnnularLinks}, as a specific instance of a broader connection between Khovanov-- and Heegaard Floer--type homology theories, using a version of Heegaard Floer homology for sutured manifolds developed by \Juhasz in \cite{MR2253454}.  In the present work we prove the naturality of the spectral sequence under certain elementary TQFT operations, using a generalization of Juh{\'a}sz's surface decomposition theorem valid for decomposing surfaces geometrically disjoint from an imbedded framed link. 
\end{abstract}
\maketitle

\section{Introduction}
Let $\bL \subset S^3$ be a link.  There is an algebraic connection, discovered by \Ozsvath and Szab{\'o}, between the Khovanov homology \cite{MR1740682} of $\bL$ and the Heegaard Floer homology \cite{MR2113019} of the double-branched cover of $\bL$.  Specifically, in  \cite{MR2141852}, Ozsv{\'a}th-Szab{\'o} construct a spectral sequence whose $E^2$ term is $\widetilde{Kh}(\overline{\bL})$ and whose $E^\infty$ term is $\widehat{HF}(\boldSigma(S^3, \bL))$.   Here (and throughout), $\widetilde{Kh}$ denotes Khovanov's reduced homology \cite{MR2034399}, $\overline{\bL}$ denotes the mirror of $\bL$, $\boldSigma(A, B)$ denotes the double-branched cover of $A$ branched over $B$, and $\widehat{HF}$ denotes the ($\wedge$ version of the) Heegaard Floer homology \cite{MR2113019}.  Unless explicitly stated otherwise, all Khovanov and Heegaard Floer homology theories discussed in this paper will be considered with coefficients in $\Z_2$.  

Later work, of Roberts in \cite{GT07060741} and the authors in \cite{GT08071432}, placed Ozsv{\'a}th-Szab{\'o}'s work in a more general context, leading to:
\begin{itemize}
  \item a proof, in \cite{GT08071432}, that Khovanov's categorification, \cite{MR2124557}, of the reduced, $n$--colored Jones polynomial detects the unknot whenever $n \geq 2$, as well as
  \item a new method, due to Baldwin-Plamenevskaya \cite{GT08082336}, for establishing the tightness of certain contact structures.
\end{itemize}

In \cite{AnnularLinks}, we recast \cite{MR2141852}, \cite{GT07060741}, and \cite{GT08071432} as specific instances of a broader relationship between Khovanov- and Heegaard Floer-type homology theories, using a version of Heegaard Floer homology for sutured manifolds developed by \Juhasz in \cite{MR2253454}.  
The aim of the present work is to prove that the connection between Khovanov and Heegaard-Floer homology 
behaves well under certain natural geometric operations.

In particular, let $D$ represent an oriented disk, $A$ an oriented annulus, and $I = [0,1]$ the oriented closed unit interval.  In \cite{GT08071432} we prove the existence of a spectral sequence from the Khovanov homology of any {\em admissible balanced tangle}, $T \subset D \times I$,  to the sutured Floer homology of $\boldSigma(D \times I, T)$.  Here, $D \times I$ is viewed as a product sutured manifold in the sense of Gabai \cite{MR723813} (Definition \ref{defn:suturedman}), sutured Floer homology \cite{MR2253454} is an invariant of balanced sutured manifolds (see Definitions \ref{defn:suturedman} and \ref{defn:balancedSM}), and an {\em admissible $n$--balanced tangle} (Definition \ref{defn:balancedtangle}) is a properly-imbedded $1$--manifold satisfying:
\begin{enumerate}
  \item $T \cap (\partial D \times I) = \emptyset$, and
  \item $|T \cap (D \times \{1\})| = |T \cap (D \times \{0\})|= n \in \Z_{\geq 0}$,
\end{enumerate}
where two admissible $n$--balanced tangles are considered equivalent if they are ambiently isotopic through admissible $n$--balanced tangles.  In \cite{AnnularLinks}, we prove the existence of a similar spectral sequence from the Khovanov homology of a link ({\em admissible $0$--balanced tangle}), $\bL$, in the product sutured manifold $A \times I$ to the sutured Floer homology of $\boldSigma(A \times I, \bL)$.  

These spectral sequences are constructed, following \cite{MR2141852}, by associating to an enhanced projection (diagram) of $T \subset D \times I$ (resp., $\bL \subset A \times I$) a framed link, $L_\bL \subset \boldSigma(D \times I, T)$ (resp., $L_\bL \subset \boldSigma(A \times I, \bL)$).  By counting holomorphic polygons in a particular choice of Heegaard multi-diagram compatible with $L_T$ (resp., $L_\bL$), one obtains a filtered complex, $X(L_T)$ (resp., $X(L_\bL)$), with an 
associated {\em link surgeries spectral sequence} 
whose $E^2$ term is an appropriate version of Khovanov homology for $T$ (resp., for $\bL$) and whose $E^\infty$ term is the sutured Floer homology of $\boldSigma(D \times I, T)$ (resp., of $\boldSigma(A \times I, \bL)$).  

Roberts, in \cite[Sec. 7]{GT08082817}, proves that the filtered quasi-isomorphism type (Definition \ref{defn:quasiiso}) of $X(L_T)$ (resp., $X(L_\bL)$) is invariant of the choice of multi-diagram, and Baldwin, in \cite{GT08093293}, proves that the filtered quasi-isomorphism type of $X(L_T)$ is independent of the projection of $T$ (resp., $\bL$), yielding, for any $n$--balanced tangle $T \subset D \times I$ (resp., any link $\bL \subset A \times I$) a sequence of invariants, one for every page of the link surgeries spectral sequence for $L_T$ (resp., $L_\bL$).\footnote{Baldwin and Roberts state their theorems only for the case where $T \subset D \times I$ is a $1$--balanced tangle, but their arguments are all local, hence work equally well in our more general setting.  See Remark \ref{rmk:SpecSeqInv}.}  

In the present work, we show that these invariants behave ``as expected'' with respect to the following standard TQFT-type operations:
\begin{enumerate}
  \item trivial inclusion (see Figure \ref{fig:AdjoinTrivial}),
  \item horizontal stacking (see Figure \ref{fig:Stacking}), and
  \item vertical cutting (see Figure \ref{fig:TQFT}).
\end{enumerate}

\begin{figure}
\begin{center}
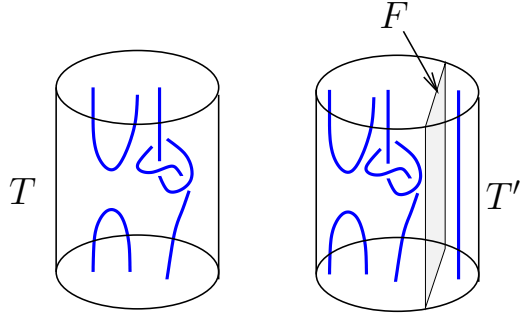
\end{center}
\caption{Adjoining a trivial strand, separated from an $n$--balanced tangle, $T \subset D \times I$, by a vertical disk, $F$, to form an $n+1$--balanced tangle, $T' \subset D \times I$.}
\label{fig:AdjoinTrivial}
\end{figure}

\begin{figure}
\begin{center}
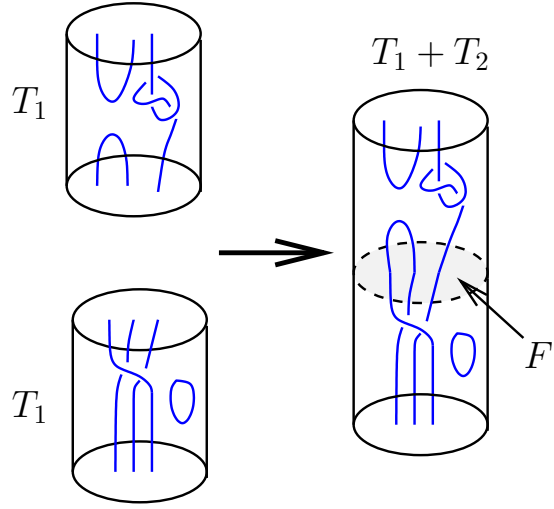
\end{center}
\caption{Stacking two (projections of) balanced tangles, $T_1, T_2 \subset D \times I$, to obtain a new balanced tangle, $T_1 + T_2 \subset D \times I$.}   
\label{fig:Stacking}
\end{figure}

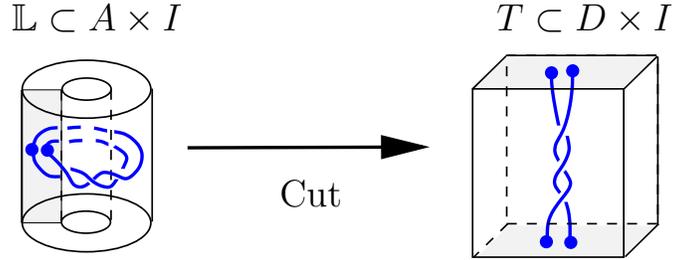
\begin{figure}
\begin{center}
\input{Figures/TQFT.pstex_t}
\end{center}
\caption{Cutting a link, $\bL \subset A \times I$, to obtain a balanced tangle, $T \subset D \times I$.}
\label{fig:TQFT}
\end{figure}

In particular, let $\cF(T):=X(L_T)$ (resp., $\cF(\bL):=X(L_\bL)$) denote the filtered chain complex, described above and in Notation \ref{defn:FiltCpxTangle}, associated to the balanced tangle $T \subset D \times I$ (resp., link $\bL \subset A \times I$).  We prove:

\begin{theoremAdjoinTriv}(Trivial inclusion) Let $T \subset D \times I$ be a balanced tangle in the product sutured manifold $D \times I$, and let $T' \subset D \times I$ be the tangle obtained from $T$ by adjoining a trivial strand separated from $T$ by a properly-imbedded $I$--invariant disk, $F$, as in Figure \ref{fig:AdjoinTrivial}.  
Then \[\cF(T) = \cF(T').\]
\end{theoremAdjoinTriv}

\begin{theoremStack} (Stacking) Let $T_i \subset (D \times I)_i$, for $i=1,2$, be two $n$--balanced tangles,  and let $T_1 + T_2 \subset D \times I$ be any $n$--balanced tangle obtained by stacking a projection, $\cP(T_1)$, of $T_1$ on top of a projection, $\cP(T_2)$, of $T_2$ as in Figure \ref{fig:Stacking}.  Then \[\cF(T_1 + T_2) = \cF(T_1) \otimes \cF(T_2).\]
\end{theoremStack}

\begin{theoremCut}(Cutting) \cite[Thm. 3.1]{AnnularLinks} Let $\bL \subset A \times I$ be a link, and let $T \subset D \times I$ be any balanced tangle admitting a projection whose closure in $A \times I$ is a projection of $\bL$.  Then \[\cF(T) \leq \cF(\bL).\]
\end{theoremCut}

  In the above, ``$\cF_1 = \cF_2$'' means that $\cF_1$ is  filtered quasi-isomorphic (Definition \ref{defn:quasiiso}) to $\cF_2$, and ``$\cF_1 \leq \cF_2$'' means that $\cF_1$ is filtered quasi-isomorphic to a direct summand of $\cF_2$.

These naturality theorems follow from an extension of work of Juh{\'a}sz, who proves, in \cite{MR2390347}, that the Floer homology of a sutured manifold behaves nicely in the presence of admissible {\em decomposing surfaces} (Definition \ref{defn:DecSurf}), properly-imbedded surfaces with boundary intersecting the sutures in a controlled fashion.  In particular, an admissible decomposing surface, $S$, induces a splitting of the sutured Floer chain complex, and performing a {\em surface decomposition} (Definition \ref{defn:DecSurf}) along $S$ picks out a direct summand of the splitting.

Using degeneration techniques suggested to us by Robert Lipshitz, we prove a generalized version of Juh{\'a}sz's surface decomposition theorem, applicable to filtered chain complexes arising from sutured multi-diagrams.  More precisely:

\begin{theoremA}
Let $L \subset (Y,\Gamma)$ be a framed link in a strongly-balanced 
 sutured manifold (Definition \ref{defn:balancedSM}), and let $S \subset (Y,\Gamma)$ be a connected decomposing surface satisfying:
\begin{enumerate}
  \item $S \cap L = \emptyset$, and
  \item for every component $V$ of $R(\Gamma)$ the closed components of the intersection $V \cap S$ consist of parallel oriented boundary-coherent curves (Definition \ref{defn:boundcoh}).
\end{enumerate}

Let $(Y', \Gamma')$ be the sutured manifold obtained by decomposing along $S$ and $L' \subset (Y',\Gamma')$ the induced image of $L$.  If $X(L)$ (resp., $X(L')$) is the filtered complex associated to $L$ (resp., $L'$), then \[X(L') \leq X(L).\]


\end{theoremA}

The paper is organized as follows:
\begin{itemize}
  \item In Section \ref{sec:Notation}, we fix notation.
  \item In Section \ref{sec:LinkSurg}, we discuss splitting of the link surgeries spectral sequence in the presence of a decomposing surface disjoint from the link.
  \item In Section \ref{sec:SurfDecomp}, we prove the surface decomposition theorem for sutured multi-diagrams (Theorem \ref{thm:SurfDecomp}).
  \item In Section \ref{sec:Naturality}, we prove the naturality results (Theorems \ref{thm:AdjoinTrivial}, \ref{thm:Stack}, and \ref{thm:Cutting}).  In this last section, we also discuss the relationship between the stacking operation and a generalized version of the Murasugi sum.
\end{itemize}

{\flushleft {\bf Acknowledgments:}} We thank John Baldwin, Matt Hedden, Nathan Habegger, Andr{\'a}s Juh{\'a}sz, Mikhail Khovanov, Rob Kirby, Robert Lipshitz, Peter Ozsv{\'a}th, Lawrence Roberts, and Liam Watson for interesting conversations.  We are particularly indebted to Robert Lipshitz for providing us with both the key idea in the proof of Theorem \ref{thm:SurfDecomp} and extremely valuable feedback on a preliminary draft.  A portion of this work was completed while the second author was a visiting postdoctoral fellow at Columbia University, supported by a Swiss NSF fellowship for prospective researchers.  We are grateful to the Columbia mathematics department for its hospitality.

\section{Definitions and Notation Conventions} \label{sec:Notation}
We shall assume familiarity with \cite{GT08071432} and \cite{AnnularLinks}, where most relevant background material and notation is collected.  See, in particular:
\begin{itemize}
  \item \cite[Sec. 2]{GT08071432}, which contains necessary definitions and results related to sutured manifolds and sutured Floer homology collected from \cite{MR723813}, \cite{MR2253454}, and \cite{MR2390347};
  \item \cite[Sec. 4]{GT08071432}, which develops Ozsv{\'a}th-Szab{\'o}'s link surgeries spectral sequence (\cite{MR2141852}) for framed links in sutured manifolds;
  \item \cite[Sec. 5]{GT08071432} (resp, \cite[Sec. 2]{AnnularLinks}), which explains how to associate to an admissible balanced tangle, $T$ (resp., link $\bL$) in the product sutured manifold, $D \times I$ (resp, $A \times I$), two chain complexes and a spectral sequence from the first to the second, one arising from a Khovanov-type construction and the other arising from Juh{\'a}sz's Heegaard Floer-type construction applied to $\boldSigma(D \times I, T)$ (resp., $\boldSigma(A \times I, \bL)$).
\end{itemize}

We repeat the following definitions and establish the following notation for the convenience of the reader:
\begin{definition} \cite{MR723813} \label{defn:suturedman}
A {\em sutured manifold} $(Y,\Gamma)$ is a compact, oriented $3$--manifold with boundary $\partial Y$ along with a set $\Gamma \subset \partial Y$ of pairwise disjoint annuli $A(\Gamma)$ and tori $T(\Gamma)$.  The interior of each component of $A(\Gamma)$ contains a {\em suture}, an oriented simple closed curve which is homologically nontrivial in $A(\Gamma)$.  The union of the sutures is denoted $s(\Gamma)$.

Every component of $R(\Gamma) = \partial Y - \mbox{Int}(\Gamma)$ is assigned an orientation compatible with the oriented sutures.  
Let $R_+(\Gamma)$ (resp., $R_-(\Gamma)$) denote those components of $R(\Gamma)$ whose normal vectors point out of (resp., into) $Y$.
\end{definition}

\begin{definition} \cite[Defn. 2.2]{MR2253454} \label{defn:balancedSM} A sutured manifold $(Y,\Gamma)$ is said to be {\em balanced} if $\chi(R_+) = \chi(R_-)$, and the maps $\pi_0(\Gamma) \rightarrow \pi_0(\partial Y)$ and $\pi_0(\partial Y) \rightarrow \pi_0(Y)$ are surjective.\footnote{The equivalence of this definition to the original definition in \cite{MR2253454} is immediate.}

A sutured manifold $(Y,\Gamma)$ is said to be {\em strongly balanced} if for each component $F$ of $\partial Y$, $\chi(F \cap R_+(\Gamma)) = \chi(F \cap R_-(\gamma)$.
\end{definition}

\begin{definition} A {\em product sutured manifold} is a sutured manifold of the form $(F \times I, \partial F \times I)$, where $F$ is an oriented surface with $\partial F \neq \emptyset$ and $I = [0,1]$ is the closed unit interval.
\end{definition}

\begin{definition} \cite[Defn. 5.1-5.2]{GT08071432} \label{defn:balancedtangle} Let $F$ be an oriented surface with $\partial F \neq \emptyset$, and let $(F \times I, \partial F \times I)$ be the associated product sutured manifold, with $F_+:= F \times \{1\}$ (resp., $F_-:= F \times \{0\}$).  

An {\em admissible $n$--balanced tangle}, $T \subset F \times I$ is (any representative of) an equivalence class of properly-imbedded (unoriented) $1$--manifolds satisfying

\begin{itemize}
  \item $T \cap \partial(F \times I) \subset \mbox{Int} (F_+) \cup \mbox{Int} (F_-)$,
  \item $T_1, T_2$ are equivalent if they can be connected by an ambient isotopy acting trivially on $\partial F \times I$,
  \item $|T \cap F_+| = |T \cap F_-| = n \in \Z_{\geq 0}$.
\end{itemize}

We will often refer to $0$--balanced tangles as links.
\end{definition}

Let $D$ represent an oriented disk and $A$ an oriented annulus.  In the present work, we will focus on admissible balanced tangles $T \subset D \times I$ and links $\bL \subset A \times I$.





\begin{notation} \label{defn:SpecSeqFiltCpx} Let $\cF$ be a filtered chain complex.  We will denote by $ss(\cF)$ the spectral sequence induced by $\cF$ and $E^{i}(\cF)$ its $i$--th page.
\end{notation}

\begin{definition}\label{defn:quasiiso} Let $\cF_1$, $\cF_2$ be two filtered chain complexes.  We shall say that $\cF_1$, $\cF_2$ are filtered quasi-isomorphic if there exists a third filtered chain complex, $\cF'$, and filtered chain maps \[\phi_j: \cF_j \rightarrow \cF',\] such that \[\phi_j: E^i(\cF_j) \rightarrow E^i(\cF')\] is an isomorphism for all $i \in \Z_+$, $j = 1,2$. 

If $\cF_1$ and $\cF_2$ are filtered quasi-isomorphic, we shall say \[\cF_1 = \cF_2.\]
\end{definition}

\begin{notation} \label{defn:FiltCpxTangle} 
Let $D$ denote an oriented disk, $A$ an oriented annulus, and $I = [0,1]$ the oriented closed unit interval. 

Let $T \subset D \times I$ (resp., $\bL \subset A \times I$) be an admissible balanced tangle (resp., link) in the product sutured manifold $D \times I$ (resp., $A \times I$), and let $L_T \subset \boldSigma(D \times I, T)$ (resp., $L_\bL \subset \boldSigma(A \times I, \bL)$) be its associated surgery link in the double-branched cover, constructed as in \cite[Sec. 5]{GT08071432} (resp., \cite[Sec. 2]{AnnularLinks}) by taking the preimage of simple arcs at each crossing of the projection of $T \subset D \times I \subset \R^3$ to the $xz$ plane (resp., of $\bL \subset A \times I$ to $A$).

We will denote by $\cF(T)$ (resp., $\cF(\bL)$) the filtered complex inducing the link surgeries spectral sequence associated to $L_T$ (resp., $L_\bL$).  Recall that $E^2(\cF(T))$ (resp., $E^2(\cF(\bL))$) is an appropriate version of the Khovanov homology of $T$ (resp., $\bL$) and $E^\infty(\cF(T))$ (resp., $E^\infty(\cF(\bL))$ is $SFH(\boldSigma(D\times I,T))$ (resp., $SFH(\boldSigma(A \times I, \bL))$.  Note that $\cF(T)$ (resp., $\cF(\bL)$) is well-defined, up to filtered quasi-isomorphism (see Remark \ref{rmk:SpecSeqInv}).
\end{notation}


In addition, we note:

\begin{enumerate}
  \item Unless explicitly stated otherwise, every sutured manifold encountered in the present paper may be assumed to be {\em strongly balanced} with no toroidal sutures (i.e., $T(\Gamma) = \emptyset$).  Each sutured manifold will furthermore be endowed with a Riemannian metric, along with a canonical unit vector field along $\partial Y$, denoted $v_0$ (Notation \ref{not:v0}).
  \item If $(Y,\Gamma)$ is a ``standard'' sutured manifold of one of the types describe in \cite[Ex. 2.5-2.7]{GT08071432} or obtained from such a sutured manifold by taking a cyclic branched cover over an admissible, properly imbedded $1$--manifold as described in \cite[Defn. 2.11]{GT08071432}, we will sometimes omit reference to $\Gamma$ in the notation.
  \item $SFH(Y)$ will denote the {\em sutured Floer homology} \cite{MR2253454} of $(Y,\Gamma)$, and $SFH(Y;\spincs)$ will denote the sutured Floer homology of $(Y,\Gamma)$ in the Spin$^c$ structure $\spincs$.  Recall that a Spin$^c$ structure on a sutured manifold is a homology class of unit vector fields on $(Y,\Gamma)$, all of which agree with $v_0$ on $\partial Y$ (see Definition \ref{defn:spinc}).  $SFH(Y)$ is the homology of any chain complex, $CFH(Y)$, obtained by applying the standard Heegaard Floer construction to the two half-dimensional tori, $\Torus_\alpha, \Torus_\beta \subset Sym^d(\Sigma)$, associated to a balanced Heegaard diagram, $(\Sigma, \boldalpha,\boldbeta)$, for $(Y,\Gamma)$.  In particular, generators of $CFH(Y)$ are elements ${\bf x} \in \Torus_\alpha \cap \Torus_\beta$.  See \cite{MR2253454} for more details.
\end{enumerate}

\section{Splitting the Link Surgeries Spectral Sequence}\label{sec:LinkSurg}
In \cite[Sec. 4]{GT08071432} (following \cite[Sec. 4]{MR2141852}), it was proved that an oriented, framed link \[L = L_1 \cup \ldots \cup L_\ell\] in a balanced, sutured manifold $(Y,\Gamma)$ induces a link surgeries spectral sequence which converges to $SFH(Y,\Gamma)$, the sutured Floer homology of $(Y,\Gamma)$. In this section, we prove that this spectral sequence splits as a direct sum of spectral sequences in the presence of a properly-imbedded surface $(S,\partial S) \subset (Y,\partial Y)$ satisfying $L \cap S = \emptyset$.  This generalizes \cite[Sec. 7]{GT07060741}, which treats the case of a Seifert surface in a sutured knot complement.

\subsection{Sutured multi-diagrams, Spin$^c$ structures, and Alexander gradings}\label{subsec:Agradings}

In order to set up the statement of the result, we recall some definitions.

\begin{notation} \label{not:v0} \cite[Not. 3.1]{MR2253454} Let $v_0$ be a non-zero vector field along $\partial Y$ that points into $Y$ along $R_-(\Gamma)$, points out of $Y$ along $R_+(\Gamma)$, and on $\Gamma$ is the gradient of the height function $s(\Gamma) \times I \rightarrow I$.
\end{notation}

\begin{definition} \label{defn:spinc} \cite[Sec. 4]{MR2253454} A {\em Spin$^c$ structure}, $\mathfrak{s}$, on a sutured manifold, $(Y,\Gamma)$ is a homology class of nowhere-vanishing vector fields on $Y$, all agreeing with $v_0$ on $\partial Y$.  Two non-vanishing vector fields are said to be homologous if, away from finitely many points in $\mbox{Int}(Y)$, they are homotopic through non-vanishing vector fields, rel $\partial Y$.  The set of Spin$^c$ structures on $(Y,\Gamma)$ will be denoted Spin$^c(Y,\Gamma)$.
\end{definition}

\begin{definition} \cite[Defn. 3.6]{MR2390347} If $\spincs \in$ Spin$^c(Y,\Gamma)$, and $t$ is a trivialization of $v_0^\perp$, then $c_1(\mathfrak{s},t)$ is the relative Euler class of $(v_0)^\perp$ with respect to $t$.
\end{definition}

\begin{remark}Note that if $(Y,\Gamma)$ is strongly-balanced, then $(v_0)^\perp$ is trivializable (\cite[Prop. 3.4]{MR2390347}).\end{remark}

With these notions in place, we see that a properly-imbedded surface, $(S,\partial S) \subset (Y,\partial Y)$, in a balanced, sutured $3$--manifold, $(Y,\Gamma)$, induces a splitting of any chain complex, $CFH(Y)$, used to compute $SFH(Y)$.  In particular, let $[S] \in H_2(Y,\partial Y;\Z)$ denote the homology class of $S$ and define: 
\[ \spincs_k(S) := \{\spincs \in Spin^c(Y,\Gamma) \,\,|\,\,\langle c_1(\spincs,t),[S]\rangle = 2k.\}\]

Subject to the map $\spincs: \Torus_\alpha \cap \Torus_\beta \rightarrow Spin^c(Y,\Gamma)$ defined in \cite[Sec. 4]{MR2253454} (following \cite[Sec. 2.6]{MR2113019}), $[S]$ endows the generators of $CFH(Y)$ with a $\frac{1}{2}\Z$--grading, which we will call the {\it $\mbox{Alexander}_S$--grading}, or ${\bf A}_S$--grading, for short.  In particular, if ${\bf x} \in \Torus_\alpha \cap \Torus_\beta$ is a generator of $CFH(Y)$, then we define \[{\bf A}_S({\bf x}):= \frac{1}{2}\langle c_1(\spincs({\bf x}),t),[S]\rangle.\]  Since the differential on the complex is Spin$^c$-structure preserving, it is, in particular, ${\bf A}_S$--grading preserving.  Hence, $[S]$ induces a decomposition of $SFH(Y)$:
\begin{eqnarray*}
SFH(Y) &=& \bigoplus_{k \in \frac{1}{2}\Z} SFH(Y;\spincs_k(S)).\\
\end{eqnarray*}

In other words, \[{\bf A}_S({\bf x}) = k \Longleftrightarrow \spincs({\bf x}) \in \spincs_k(S).\]

\begin{remark}  Gradings induced on sutured Floer chain complex generators by properly-imbedded surfaces were defined by \Juhasz in \cite[Sec. 4]{GT08023415}.  Note that changing the trivialization induces an overall shift in the ${\bf A}_S$ gradings by half the rotation number around $\partial S$ of the new trivialization with respect to the old, by \cite[Lem. 3.11]{GT08023415}.\end{remark}


Now, suppose $L = L_1 \cup \ldots \cup L_\ell$ is an oriented,  framed link in a sutured manifold $(Y,\Gamma)$.  Denote by $(\lambda_1, \ldots, \lambda_\ell)$ the $\ell$--tuple of framings and $(\mu_1, \ldots, \mu_\ell)$ the $\ell$--tuple of meridians.  Then to each ``multi-framing,'' \[\cI = (m_1, \ldots, m_\ell) \in \{0,1,\infty\}^\ell,\] in the sense of \cite[Sec. 4]{MR2141852}, we can associate a sutured manifold, $Y_{\cI}$, obtained by doing $m_i$--framed surgery on $L_i$ for each $i = 1, \ldots, \ell$.   Here $\infty$ means no surgery, $0$ means $\lambda_i$--framed surgery, and $1$ means $(\lambda_i + \mu_i)$--framed surgery.

Beginning with the data of a bouquet $L\cup a_1\cup\ldots\cup a_{\ell}$ (where $a_i$ is an embedded arc in $Y$ whose interior is disjoint form $L$ and from $a_j$ for $j\neq i$, and which connects a point on $L_i$ to a point on $R_+(\Gamma)$), we can construct a sutured Heegaard diagram $(\Sigma,\boldalpha,\boldbeta_{\cI})$ representing a sutured manifold, $\YI$, homeomorphic to $Y_{\cI}$, as follows. Let \[(\Sigma,\{\alpha_1,\ldots,\alpha_d\},\{\beta_{\ell+1},\ldots \beta_d\})\] be a sutured Heegaard diagram for $Y-L'$, where \[L':=N(L\cup a_1\cup\ldots\cup a_{\ell})\] is a regular neighborhood of the bouquet, and let $\gamma_1,\ldots, \gamma_{\ell}\subset \partial(Y-L')$ be disjoint simple closed curves specifying the multi-framing $\cI=(m_1 \ldots,m_{\ell})$. If we set $(\beta_{\cI})_i:=\gamma_i$ for $i\leq \ell$ and $(\beta_{\cI})_i:=\beta_i$ for $\ell < i \leq d$, then $(\Sigma,\boldalpha,\boldbeta_{\cI})$ is a Heegaard diagram representing the sutured manifold, $\YI$, obtained from $Y-L'$ by attaching $3$-dimensional $2$-handles along the curves $\gamma_i\subset Y-L'$.  
Furthermore, there is a diffeomorphism\footnote{when considered as a map between smooth manifolds with corners, cf. \cite{MR1930091}.}
$$
f_{\cI}\colon Y_{\cI}\longrightarrow Y(\cI),
$$
which restricts to the identity map on $Y_{\cI}-L''=Y-L''$, where $L''$ denotes a suitably chosen thickening of the regular neighborhood $L'$.  In particular, this implies that $Y_{\cI}$ and $\YI$ are equivalent as sutured manifolds and that we have a canonical inclusion $i_{\cI}:=f_{\cI}|_{Y-L''}$ of $Y-L''$ into $Y(\cI)$ for each $\cI \in \{0,1,\infty\}^\ell$. Moreover, $f_{\cI}|_{\partial Y_{\cI}}$ agrees with the identity map on $(\partial Y_{\cI})-L''$, and since $L''\cap\partial Y_{\cI}$ is a disjoint union of disks, this determines $f_{\cI}|_{\partial Y_{\cI}}$ uniquely up to isotopy relative to $(\partial Y_{\cI})-L''$.

\begin{definition}\label{defn:successor}
We endow the set $\{0,1,\infty\}^{\ell}$ with the dictionary order, and we call a tuple $\cI' = (m_1', \ldots, m_\ell') \in \{0,1,\infty\}^\ell$ an {\it immediate successor} of $\cI= (m_1, \ldots, m_\ell)$ if there exists some $j$ such that $m_i = m_i'$ when $i\neq j$ and $(m_j,m_j')$ is either $(0,1)$ or $(1,\infty)$.
\end{definition}



It is explained in \cite[Sec. 4]{GT08071432} (following \cite[Sec. 4]{MR2141852}) how to construct a sutured Heegaard multi-diagram \[\left(\Sigma,\boldalpha,\boldbeta_{\cI_{i_1}}, \ldots, \boldbeta_{\cI_{i_m}}\right)\] for every subset \[\{\cI_{i_1}, \ldots, \cI_{i_m}\} \subseteq \{0,1,\infty\}^\ell,\] beginning with the data of a bouquet subordinate to $L$.

\begin{notation} \label{defn:HDforsubset}
Let $\{\cI_{i_1}, \ldots, \cI_{i_m}\} \subseteq \{0,1,\infty\}^\ell$ be any nonempty subset.  We denote by $\left(\Sigma, \boldalpha, \boldbeta_{\{\cI_{i_1}, \ldots, \cI_{i_m}\}}\right)$ any sutured multi-diagram compatible with $L$, produced by the method outlined in \cite[Sec. 4]{GT08071432}.  I.e., \[\left(\Sigma, \boldalpha, \boldbeta_{\{\cI_{i_1}, \ldots, \cI_{i_m}\}}\right):=\left(\Sigma, \boldalpha,\boldbeta_{\cI_{i_1}}, \ldots, \boldbeta_{\cI_{i_m}}\right).\]  In particular,  for each $\cI \in \{\cI_{i_1}, \ldots, \cI_{i_m}\}$,  $(\Sigma, \boldalpha, \boldbeta_{\cI})$ is a sutured Heegaard diagram for $\YI$.
\end{notation}

We will have particular interest in $\{0,1,\infty\}^\ell \subseteq \{0,1,\infty\}^\ell$ and $\{0,1\}^\ell \subset \{0,1,\infty\}^\ell$.  Accordingly:

\begin{definition} \label{defn:HDforlink} Given a framed $\ell$--component link $L$, we call any $\left(\Sigma, \boldalpha, \boldbeta_{\{0,1,\infty\}^\ell}\right)$ a {\em full sutured multi-diagram for L} and $\left(\Sigma, \boldalpha, \boldbeta_{\{0,1\}^\ell}\right)$ a {\em $(0,1)$ sutured multi-diagram for L}.
\end{definition} 

From these sutured multi-diagrams, one obtains chain complexes:

\begin{itemize}
\item 
From a full sutured multi-diagram for $L$, one constructs a chain complex \[X = \bigoplus_{\cI \in \{0,1,\infty\}^\ell} CFH(\YI),\] where $CFH(\YI)$ is the chain complex associated to $(\Sigma, \boldalpha,\boldbeta_{\cI})$.  The differential \[D: X \rightarrow X\] is specified, on a generator $\xi \in CFH(\YI)$, by \[D\xi := \sum_{\mathcal{J}} \sum_{\{\cI = \cI^0<\ldots<\cI^j = \mathcal{J}\}} D_{\cI^1<\ldots<\cI^j}(\xi).\] Here the index set of the inner sum is taken over the set of all increasing sequences $\cI$ to $\mathcal{J}$ having the property that $\cI^{i+1}$ is an immediate successor of $\cI^i$, and $D_{\cI^0< \ldots < \cI^j}$ is defined by counting certain holomorphic $j+2$--gons in the sutured multi-diagram, $(\Sigma, \boldalpha, \boldbeta_{\cI^1}, \ldots, \boldbeta_{\cI^j})$.\\
\item 
From a $(0,1)$ sutured multi-diagram for $L$, one similary constructs the complex $X^{(0,1)} \subset X$, with restricted differential $D^{(0,1)}$.  It is then proved, in \cite[Prop. 4.1]{GT08071432}, that $X^{(0,1)}$ is a filtered chain complex, whose associated spectral sequence has $E^1$ term \[\bigoplus_{\cI \in \{0,1\}^\ell} SFH(\YI),\] and $E^\infty$ term $SFH(Y(\cI_\infty))$, where $\cI_\infty := (\infty, \ldots, \infty)$.
\end{itemize}

\begin{remark} \label{rmk:SpecSeqInv} By work of Roberts in \cite[Sec. 7]{GT08082817}, we know that the filtered quasi-isomorphism type of $X$ (resp., $X^{(0,1)}$) is independent of the choice of full (resp., $(0,1)$) sutured multi-diagram for $L$ and, hence, each term of the associated spectral sequence is an invariant of the framed link $L \subset (Y,\Gamma)$. 

More specifically, any two full or $(0,1)$ sutured multi-diagrams compatible with a given framed link $L \subset (Y,\Gamma)$ are related by a sequence of (de)stabilizations, isotopies, and handleslides.  Roberts, in \cite[Sec. 7]{GT08082817}, constructs a filtered chain map associated to each type of move, which he proves is a filtered quasi-isomorphism.\footnote{More precisely, Roberts proves the existence of a ``$1$--quasi-isomorphism,'' a filtered chain map inducing an isomorphism on the $E^1$ page (and, hence, an isomorphism on every subsequent page) of the associated spectral sequence.}  In the case of a (de)stabilization, this map is the obvious filtered chain isomorphism \cite[Sec. 10]{MR2113019}, while in the case of an isotopy (resp., handleslide), the map is obtained by counting holomorphic polygons in a sutured multi-diagram containing the curves before and after the isotopy (resp., handleslide).  To prove that each map is a filtered chain map, he uses polygon associativity \cite[Sec. 4]{MR2141852}, verifying in each case that certain ends of $1$--dimensional moduli spaces of polygons cancel in pairs as in \cite[Lem. 4.5]{MR2141852}.  Since each map agrees with the one defined by Ozsv{\'a}th-Szab{\'o} in the proof of the isotopy (resp., handleslide) invariance of Heegaard Floer homology (see \cite{MR2113019}, \cite{MR2253454}), it induces an isomorphism on the $E^1$ page.  

Furthermore, if a link $L_T \subset \boldSigma(D\times I, T)$ (resp., $L_\bL \subset \boldSigma(A \times I, \bL)$) is obtained from an admissible balanced tangle $T \subset D \times I$ (resp., a link $\bL \subset A \times I$) as in \cite{GT08071432} (resp., \cite{AnnularLinks}), then each page of the link surgeries spectral sequence associated to $L_T$ (resp., $L_\bL$) is an invariant of $T$ (resp., $\bL$).  This follows immediately from \cite{GT08093293}, in which Baldwin proves that if two projections of a link $\bL$ in $S^3$ are related by a sequence of Reidemeister moves, then the resulting link surgeries spectral sequence for $L_\bL$ is an invariant of $\bL$.  His arguments transfer to our more general setting without change.  We need only note:
  \begin{enumerate}
    \item Any two projections of a balanced tangle $T \subset D \times I$ are related by projection isotopies, Reidemeister moves, and insertion/deletion of a braid at either the top or bottom.  Projection isotopies do not affect $L_T$, and insertion/deletion of a braid at the top or bottom results in a filtered quasi-isomorphic complex, because all resolutions of a braid have backtracking except the trivial one (see \cite[Sec. 5.2]{GT08071432}).
    \item Any two diagrams for a link $\bL \subset A \times I$ are related by diagram isotopies and Reidemeister moves which, by \cite{GT08093293}, yield filtered quasi-isomorphic complexes.
  \end{enumerate} 
\end{remark}

Bearing Remark \ref{rmk:SpecSeqInv} in mind, if $L \subset (Y,\Gamma)$ is a framed link, we will henceforth refer to {\em the} link surgeries spectral sequence associated to $L$.  Similarly, if a link $L_T \subset \boldSigma(D\times I, T)$ (resp., $L_\bL \subset \boldSigma(A \times I, \bL)$) is obtained from an admissible balanced tangle $T \subset D \times I$ (resp., a link $\bL \subset A \times I$) as in \cite{GT08071432} (resp., \cite{AnnularLinks}), then we will refer to {\em the} spectral sequence associated to $T$ (resp., $\bL$).

Now suppose $L = L_1 \cup \ldots \cup L_\ell \subset (Y,\Gamma)$ is geometrically disjoint\footnote{Note that if $L_i \cap S = 0$ algebraically, one can find a homologous surface, $S'$, which satisfies this condition.} from a given oriented, imbedded decomposing surface $(S,\partial S) \subset (Y,\partial Y)$ 
satisfying all of the hypotheses of Theorem~\ref{thm:SurfDecomp}. In view of \cite[Lem.~4.5]{MR2390347}, we can then assume that $S$ may be isotoped in a neighborhood of $\partial Y$ so that it is {\em good}, i.e. so that each component of $\partial S$ intersects both $R_+(\Gamma)$ and $R_-(\Gamma)$.  This, in turn, implies that each component of $Y-S$ intersects both $R_+(\Gamma)$ and $R_-(\Gamma)$, so we can find a bouquet $L\cup a_1\cup\ldots\cup a_{\ell}$ geometrically disjoint from $S$. Denoting by $L''$ a (thickened) regular neighborhood of $L\cup a_1\cup\ldots\cup a_{\ell}$ as in the paragraph preceding Definition~\ref{defn:successor}, we can in fact assume that $S$ is fully contained in $Y-L''$. In the following definitions, $\cI\in\{0,1,\infty\}^{\ell}$ is an arbitrary multi-framing, and $Y(\cI)$ is the sutured manifold defined in the paragraph preceding Definition~\ref{defn:successor}.

\begin{definition}\label{defn:CompSurf} Let $(S,\partial S)\subset (Y,\partial Y)$ be a decomposing surface that is fully contained in $Y-L''$. Then we denote by $\SI$ the image of $S$ under the natural imbedding $Y-L''\hookrightarrow\YI$, and we say that $\SI\subset\YI$ is {\it compatible} with $S\subset Y$.
\end{definition}

Now fix a trivialization, $t$, of the oriented $2$-plane field $(v_0)^{\perp}$ on $\partial Y$.  Recall that $Y_\cI$ is the sutured manifold obtained by performing $\cI$--surgery on $L \subset Y$, hence $\partial Y_\cI=\partial Y$ for all $\cI \in \{0,1,\infty\}^\ell$.  

\begin{definition} For each $\cI\in\{0,1,\infty\}^{\ell}$, let $f_{\cI}: Y_\cI \rightarrow \YI$ be the smooth map introduced in the paragraph preceding Definition~\ref{defn:successor} and \[df_{\cI}: TY_\cI \rightarrow T\YI\] its differential.  Then we denote by $t_{\cI}:=df_{\cI}(t)$ the trivialization of $(v_0)^\perp$ on $\partial\YI$ induced by $t$, and we say that $t_\cI$ is {\em compatible} with $t$.
\end{definition}


Given the trivializations $t_{\cI}$, we can endow each chain complex, $CFH(\YI)$, with ${\bf A}_{\SI}$--gradings and this, in turn, endows \[X = \bigoplus_{\cI \in \{0,1,\infty\}^\ell} CFH(\YI)\] with ${\bf A}_S$ gradings.  The following lemma verifies that the differential, $D$, in the complex, $X$, respects this grading.

\begin{lemma} \label{lem:FiltPres} Let $L = L_1 \cup \ldots \cup L_\ell \subset (Y,\Gamma)$ be an oriented, framed link and $S \subset (Y,\Gamma)$ a decomposing surface satisfying the conditions of Theorem \ref{thm:SurfDecomp}.   Suppose 
$\cI^1 < \ldots < \cI^k$ is a sequence of multi-framings such that $\cI^{i+1}$ is an immediate successor of $\cI^{i}$ for each $i \in \{1, \ldots, k-1\}$, and $\left(\Sigma, \boldalpha,\boldbeta_{\cI^1}, \ldots, \boldbeta_{\cI^k}\right)$ is an associated sutured multi-diagram.  Then the map $D_{\cI^1 < \ldots < \cI^k}$ obtained by counting holomorphic $(k+1)$--gons preserves ${\bf A}_{S}$--gradings.
\end{lemma}


\begin{proof}
Let $\Psi \in \pi_2({\bf x}, \boldtheta, \ldots, \boldtheta, {\bf y})$ be a domain representing a $(k+1)$--gon contributing to $D_{\cI^1< \ldots < \cI^k}$, where
\begin{enumerate}
  \item ${\bf x} \in \Torus_\alpha \cap \Torus_{\beta_{\cI^1}}$
  \item ${\bf y} \in \Torus_{\alpha} \cap \Torus_{\beta_{\cI^k}}$, and
  \item $\boldtheta$ is the canonical top-degree generator in each of $\Torus_{\beta_{\cI^i}} \cap \Torus_{\beta_{\cI^{i+1}}}$ for each $i \in \{1, \ldots, k-1\}$.
\end{enumerate}

Recall (see \cite[Sec. 3]{GT08071432}) that one associates to a sutured multi-diagram, $(\Sigma, \boldeta^0, \ldots, \boldeta^k)$, 
a $4$--manifold:
\[ W_{\eta^0, \ldots, \eta^k}:= \frac{(P_{k+1} \times \Sigma) \coprod_{i=0}^k (e_i \times U_i)}{(e_i \times \Sigma) \sim (e_i \times \partial U_i)}.\]
Here, $U_i$ is the compression body obtained from $\Sigma\times [0,1]$ by attaching a $3$-dimensional $2$-handle along each curve $\eta^i_j\times\{1\}$, for each $\eta^i_j\in\boldeta^i$. $P_{k+1}$ is a topological $(k+1)$--gon, with vertices labeled $v_i$, for $i \in \Z_{k+1}$, clockwise, and $e_i$ is the edge connecting $v_i$ to $v_{i+1}$.
Note that for each $i \in \Z_{k+1}$, the sutured manifold \[-Y_{\eta^i \eta^{i+1}}:=U_i\bigcup_{\Sigma\times\{0\}}-U_{i+1}\] sits naturally as a subset of $\partial W_{\eta^0, \ldots, \eta^k}$, so we may define:
\begin{equation} \label{eqn:Z}
Z :=\partial W_{\eta^0, \ldots, \eta^k}-\bigcup_{i\in \Z_{k}}\operatorname{Int}(-Y_{\eta^i \eta^{i+1}}).
\end{equation}

Letting $W$ denote the $4$--manifold obtained as above from the sutured multi-diagram $(\Sigma, \boldalpha, \boldbeta_{\cI^1}, \ldots, \boldbeta_{\cI^k})$, we obtain a related $4$--manifold, $\overline{W}$, by appropriately ``capping off'' certain subsets of the boundary of $W$.   We will view $\overline{W}$ as a cobordism (with corners) between $Y_{\alpha,\beta_{\cI^1}}$ and $Y_{\alpha,\beta_{\cI^k}}$.

More explicitly, let $U_{\alpha}$ (resp., $U_{\beta_{\cI^i}}$) denote the compression body obtained from $\Sigma\times[0,1]$ by attaching $3$-dimensional $2$-handles along the curves $\alpha_j\times\{1\}$ (resp., $(\beta_{\cI^i})_j\times\{1\}$).  Since $\cI^{i+1}$ is an immediate successor of $\cI^i$ for $i\in\{1,\ldots,k-1\}$, the compression body $U_{\beta_{\cI^{i+1}}}$ can be obtained from $U_{\beta_{\cI^i}}$ by performing surgery on a framed knot $K_i\subset U_{\beta_{\cI^i}}$ corresponding to the component of $L$ on which the multi-framings $\cI^i$ and $\cI^{i+1}$ disagree. More precisely, we have a diffeomorphism $f_i\colon U_{\beta_{\cI^i}}(K_i)\rightarrow U_{\beta_{\cI^{i+1}}}$,
where $U_{\beta_{\cI^i}}(K_i)$ is the result of performing surgery on $K_i\subset U_{\beta_{\cI^i}}$. Let $g_i\colon\partial U_{\beta_{\cI^i}}\rightarrow\partial U_{\beta_{\cI^{i+1}}}$ be the restriction of $f_i$ to $\partial U_{\beta_{\cI^i}}(K_i)=\partial U_{\beta_{\cI^i}}$, and let $W_i$ be the $4$-manifold obtained from $U_{\beta_{\cI^i}}\times [0,1]$ by attaching a $4$-dimensional $2$-handle along the framed knot $K_i\times\{1\}$. Then
\begin{align*}
\partial W_i &=(-U_{\beta_{\cI^i}}\times\{0\})\cup (\partial U_{\beta_{\cI^i}}\times[0,1])\cup (U_{\beta_{\cI^i}}(K_i)\times\{1\})\\
&\cong (-U_{\beta_{\cI^i}}\times\{0\})\cup (\partial U_{\beta_{\cI^i}}\times[0,1])\cup (U_{\beta_{\cI^{i+1}}}\times\{1\}),
\end{align*}
where, in the second line, $\partial U_{\beta_{\cI^i}}\times[0,1]$ is glued to $U_{\beta_{\cI^{i+1}}}\times\{1\}$ using the map $g_i\times id_{\{1\}}$. The boundary of $W_i$ contains a subset diffeomorphic to $Y_{\beta_{\cI^i},\beta_{\cI^{i+1}}}=-U_{\beta_{\cI^i}}\cup_{\Sigma\times\{0\}} U_{\beta_{\cI^{i+1}}}$ (namely the subset $(-U_{\beta_{\cI^i}}\times\{0\})\cup (\Sigma\times\{0\}\times[0,1])\cup (U_{\beta_{\cI^i}}(K_i)\times\{1\})$), and hence we can use $W_i$ to cap off the copy of $-Y_{\beta_{\cI^i},\beta_{\cI^{i+1}}}$ that appears in the boundary of $W=W_{\alpha,\beta_{\cI^1},\ldots,\beta_{\cI^k}}$ (see Figure~\ref{fig:CappingOff}). After doing this for each $i\in\{1,\ldots,k-1\}$, we obtain 
$$
\overline{W}:=W\cup W_1\cup\ldots\cup W_{k-1},
$$


and, by analogy to (\ref{eqn:Z}) above, we denote:
$$
\overline{Z} := \partial \overline{W} - \left(\operatorname{Int}\left(-Y_{\alpha,\beta_{\cI^1}}\right)\right) - \left(\operatorname{Int}\left(Y_{\alpha,\beta_{\cI^k}}\right)\right)
$$


Note that $\overline{Z}$ can be identified with the mapping cylinder of the map
$$
g\colon \partial Y_{\alpha, \beta_{\cI^1}}\longrightarrow \partial Y_{\alpha,\beta_{\cI^k}},
$$
defined as the identity map on $(\partial U_{\alpha})-(\Sigma\times\{0\})$ and as the composition
$$
g_{k-1}\circ\ldots\circ g_1\colon R_+(U_{\beta_{\cI^1}})\longrightarrow R_+(U_{\beta_{\cI^k}})
$$
on $R_+(U_{\beta_{\cI^1}}):=(\partial U_{\beta_{\cI^1}})-(\Sigma\times\{0\})$. 

Now recall (cf. \cite[Defn. 3.5]{GT08071432}):

\begin{definition}
A {\em relative Spin$^c$ structure} for a pair $(W,Z)$ with $W$ a $4$--manifold and $Z \subset W$ a closed, smoothly-imbedded submanifold is an equivalence class of pairs $(\xi,P)$ where $P$ is a finite collection of points in $W-Z$, and $\xi$ is an oriented $2$-plane field on $W-P$, agreeing with a fixed, oriented $2$--plane field, $\xi_0$, on $Z$. Two such pairs $(\xi,P)$ and $(\xi',P')$ are considered to be {\em equivalent} if there is a compact $1$-manifold $C\subset W-Z$ containing $P$ and $P'$ and with the property that $\xi|_{W-C}$ and $\xi'|_{W-C}$ are isotopic rel. $Z$. We denote by $Spin^c(W,Z)$ the set of all relative Spin$^c$ structures for the pair $(W,Z)$.
\end{definition}


In \cite[Sec. 3.1]{GT08071432} (following \cite{MR2113019}) it is shown how to associate to the domain, $\Psi \in \pi_2({\bf x}, {\bf \theta}, \ldots, {\bf \theta},{\bf y})$, a relative Spin$^c$ structure $\spincs(\Psi) \in Spin^c(W,Z)$.  Moreover, $\spincs(\Psi)$ extends uniquely to a Spin$^c$ structure $\overline{\spincs}(\Psi) \in Spin^c(\overline{W},\overline{Z})$ satisfying \[\overline{\spincs}(\Psi)|_{Y(\cI^1)} = \spincs({\bf x})\hskip 10 pt \mbox{and} \hskip 10pt \overline{\spincs}(\Psi)|_{Y(\cI^k)} = \spincs({\bf y}),\] where the fixed oriented $2$--plane field, $\overline{\xi}_0$, on $\overline{Z}\cong (\partial Y_{\alpha,\beta_{\cI^1}})\times [0,1]$ is defined as \[\overline{\xi}_0:=(v_0)^{\perp}\oplus 0.\]


The proof of the lemma is now based on the following two observations:
\begin{enumerate}
\item
The surfaces $S_{\cI^1}$ and $S_{\cI^k}$ are isotopic in $\overline{W}$ rel. $\overline{Z}$.

\item
The trivializations, $t_{\cI^1}=df_{\cI^1}(t)$ and $t_{\cI^k}=df_{\cI^k}(t)$, of $(v_0)^{\perp}$ on $\partial Y(\cI^1)$ and $\partial Y(\cI^k)$ extend naturally, via the map $g$, to a trivialization, $t_{\overline{Z}}$, of $\overline{\xi}_0$ on $\overline{Z}$. 

\end{enumerate}

Observation (1) implies that \[[S_{\cI^1}]= [S_{\cI^k}] \in H_2(\overline{W},\overline{Z};\Z),\] and observation (2), combined with the naturality of the relative first Chern class, imply that
$$c_1(\overline{\spincs}(\Psi),t_{\overline{Z}})\in H^2(\overline{W},\overline{Z};\Z)$$
restricts to
$$c_1(\spincs({\bf x}),t_{\cI^1}) \in H^2(Y(\cI^1),\partial Y(\cI^1);\Z)$$
and
$$c_1(\spincs({\bf y}),t_{\cI^k})\in H^2(Y(\cI^k),\partial Y(\cI^k);\Z).$$
Here, we use the following definition for the first Chern class of a relative spin$^c$ structure, $\spincs \in Spin^c(\overline{W},\overline{Z})$, with respect to a trivialization, $t_{\overline{Z}}$, of the fixed $2$--plane field, $\xi_0$, over $\overline{Z}$:

\begin{definition} Let $\spincs \in Spin^c(\overline{W},\overline{Z})$ be 
represented by a pair, $(\overline{\xi}, \overline{P})$, where $\overline{P}$ is a finite collection of points in $\overline{W}-\overline{Z}$, and $\overline{\xi}$ is an oriented $2$--plane field on $\overline{W}-\overline{P}$, agreeing with $\overline{\xi}_0$ over $\overline{Z}$.  Then \[c_1(\overline{\spincs},t_{\overline{Z}}) \in H^2(\overline{W},\overline{Z};\Z)\] is the relative Euler class of $\overline{\xi}$ with respect to $t_{\overline{Z}}$. It is the obstruction to extending $t_{\overline{Z}}$ to a trivialization of $\overline{\xi}$ over the (relative) $2$--skeleton of $(\overline{W},\overline{Z})$.
\end{definition} 


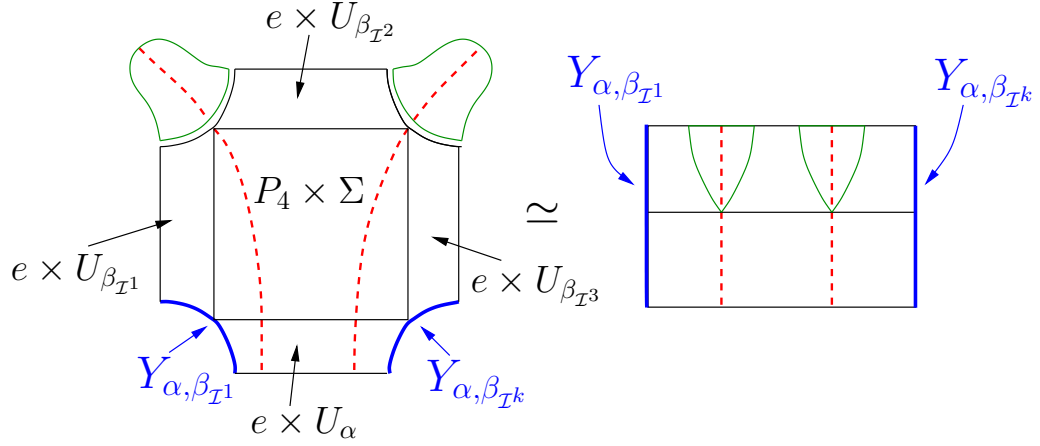
\begin{figure}
\begin{center}
\input{Figures/CappingOff.pstex_t}
\end{center}
\caption{Before capping off, the figure on the left represents the $4$--manifold $W_{\alpha,\beta_{\cI^1},\beta_{\cI^2},\beta_{cI^3}}$.  Capping off $Y_{\beta_{\cI^1},\beta_{\cI^2}}$ and $Y_{\beta_{\cI^2},\beta_{\cI^3}}$ with a copy of the restriction to the upper compression body of the appropriate $4$--dimensional $2$--handle attachment (pictured in green) yields the $4$--manifold cobordism associated to the sequence $\cI^1, \cI^2, \cI^3$, as illustrated on the right.  The red dotted lines represent the critical levels of the associated relative Morse decomposition.}
\label{fig:CappingOff}
\end{figure}

We conclude: 

\begin{eqnarray*}
{\bf A}_S({\bf x}) &=& \frac{1}{2}\langle \,c_1(\spincs({\bf x}),t_{\cI}),[S_{\cI^1}]\,\rangle\,\\
&=& \frac{1}{2}\,\langle\, c_1(\overline{\spincs}(\Psi),t_{\overline{Z}}), [S_{\cI^1}]\,\rangle\\
&=& \frac{1}{2}\,\langle\, c_1(\overline{\spincs}(\Psi),t_{\overline{Z}}), [S_{\cI^k}]\,\rangle\\
&=& \frac{1}{2}\,\langle\, c_1(\spincs({\bf y}),t_{\cI^k}),[S_{\cI^k}]\,\rangle\\
&=& {\bf A}_S({\bf y}),
\end{eqnarray*}

as desired.
\end{proof}

With this lemma in hand, it is now straightforward to construct a splitting of the link surgeries spectral sequence of \cite[Prop. 4.1]{GT08071432}, as follows.


\begin{proposition} \label{prop:SpecSeqSplit} Let $(Y,\Gamma)$ be a sutured manifold, $(S,\partial S) \subset (Y,\partial Y)$ an oriented surface, and $L = L_1 \cup \ldots \cup L_{\ell} \subset (Y,\Gamma)$ an oriented, framed link satisfying $L \cap S = \emptyset$.  
For each $m \in \frac{1}{2}\Z$, let \[X^{(0,1)}_m \subset X^{(0,1)}\] be the subcomplex of $X^{(0,1)}$ corresponding to those elements ${\bf x} \in CFH(\YI)$ satisfying \[{\bf A}_{\SI}({\bf x}) = m,\] where $\cI$ ranges over all elements in $\{0,1\}^\ell$.   

If $D^{(0,1)}_m$ denotes the restricted differential, then $(X^{(0,1)}_m,D^{(0,1)}_m)$ is a filtered chain complex whose associated spectral sequence has $E^1$ term \[\bigoplus_{\cI \in \{0,1\}^\ell} SFH(\,\YI\,\,;\,\,\spincs_m(\SI)\,)\] and whose $E^\infty$ term is \[SFH(\,Y(\cI_\infty)\,\,;\,\,\spincs_m(S)\,),\] where $\cI_\infty := (\infty, \ldots, \infty)$.
\end{proposition}

\begin{proof}
By analogy to the notation established above, let $X_m \subset X$ denote the subset of the full complex, \[X := \bigoplus_{\cI \in \{0,1,\infty\}^\ell} CFH(\YI)\] corresponding to those elements in ${\bf A}_{\SI}$--grading $m$ and let $D_m$ denote the restricted differential.  Lemma \ref{lem:FiltPres} tells us that $(X_m, D_m)$ is a subcomplex of $(X,D)$ and $\left(X^{(0,1)}_m,D^{(0,1)}_m\right)$ is a subcomplex of $\left(X^{(0,1)},D^{(0,1)}\right)$ for each $m \in \Z$.

Furthermore, recall that $\left(X^{(0,1)},D^{(0,1)}\right)$ is a filtered chain complex, where the filtration arises from a different grading on the complex.  Namely, if ${\bf x} \in CFH(\YI)$, with $\cI = (m_1, \ldots, m_\ell)$, then we say that ${\bf x}$ is in filtration-grading $\sum_{i=1}^\ell m_i$. $(X^{(0,1)},D^{(0,1)})$ is then a filtered chain complex, since $D^{(0,1)}$ is non-decreasing with respect to this grading.  This filtration-grading descends to each component, $\left(X^{(0,1)}_m, D^{(0,1)}_m\right)$, of the splitting, endowing $\left(X^{(0,1)}_m, D^{(0,1)}_m\right)$ with the structure of a filtered chain complex for each $m \in \Z$.

Now the proof follows exactly as in the proof of \cite[Prop. 4.1]{GT08071432}, following the proof of \cite[Thm. 4.1]{MR2141852}.  Lemma \ref{lem:FiltPres} tells us that the various chain maps and homotopy maps split according to ${\bf A}_{\SI}$--grading.  Furthermore, domains representing $1$--dimensional moduli spaces of holomorphic polygons, which we degenerate to prove polygon associativity (see \cite[Sec. 8.4]{MR2113019},  \cite[Sec. 4]{MR2141852}, \cite[Sec. 3.3]{GT08071432}), also respect ${\bf A}_{\SI}$--grading, by the same argument used in the proof of Lemma \ref{lem:FiltPres}.  One need only note that each pair of cancelling terms in the proof of an analogue of \cite[Lem. 4.5]{MR2141852} appears in the same ${\bf A}_{\SI}$--graded component.  The rest of the argument, applied to each ${\bf A}_{\SI}$--grading separately, is identical.
\end{proof}

\section{Surface decompositions} \label{sec:SurfDecomp}
In the previous section, we proved that the link surgeries spectral sequence associated to a framed link in a sutured manifold splits in the presence of a properly-imbedded surface whose intersection with the link is empty.  In this section, we will investigate the algebraic effect on the link surgeries spectral sequence of cutting the sutured manifold open ({\em decomposing}) along such a surface.  The main result is a generalization to multi-diagrams of Juh{\'a}sz's surface decomposition theorem (\cite[Thm. 1.3]{MR2390347}).

\subsection{Background and statement of theorem}
We recall a few definitions related to surface decompositions in sutured manifolds in preparation for stating the main theorem.

\begin{definition}\cite[Defn. 2.4]{MR2390347} \label{defn:DecSurf} A {\em decomposing surface} in a sutured manifold, $(Y,\Gamma)$, is a properly imbedded, oriented surface, $(S,\partial S) \subset (Y,\partial Y)$ such that for every component, $\lambda$, of $(\partial S) \cap \Gamma$, one of the following occurs:
  \begin{itemize}
    \item $\lambda$ is a properly imbedded non-separating arc in $\Gamma$ such that $|\lambda \cap s(\Gamma)| = 1$.
    \item $\lambda$ is a simple closed curve in an annular component, $A$, of $\Gamma$ representing the same homology class in $A$ as $s(\Gamma)$.
    \item $\lambda$ is a homotopically nontrivial curve in a torus component $T$ of $\Gamma$, and if $\delta$ is another component of $T \cap (\partial S)$, then $\lambda$ and $\delta$ represent the same homology class in $T$.
  \end{itemize}
\end{definition}

\begin{definition} Given a decomposing surface, $S$, in a sutured manifold, $Y$, the result of {\em decomposing $Y$ along $S$} is a new sutured manifold, $(Y',\Gamma')$, obtained as follows.
\begin{itemize}
  \item $Y' = Y - \mbox{Int}(N(S))$,
  \item $\Gamma' = (\Gamma \cap Y') \cup N(S_+' \cap R_-(\Gamma)) \cup N(S'_- \cap R_+(\Gamma))$,
  \item $R_+(\Gamma') = ((R_+(\Gamma) \cap Y') \cup S_+') - \mbox{Int}(\Gamma')$,
  \item $R_-(\Gamma') = ((R_-(\Gamma) \cap Y') \cup S_-') - \mbox{Int}(\Gamma')$,
\end{itemize}
where $S_+'$ (resp., $S'_-$) is the component of $\partial N(S) \cap Y'$ whose normal vector field points out of (resp., into) $Y'$. 
\end{definition}

\begin{definition}\cite[Defn. 3.7]{MR2390347} Let $(Y,\Gamma)$ be a strongly-balanced sutured manifold, $v_0$ the fixed unit vector field on $\partial Y$ as in \ref{not:v0}, $t$ a choice of trivialization of $v_0^\perp$, and $S$ a decomposing surface in $(Y,\Gamma)$.  Then \[c(S,t) := \chi(S) - \frac{I(S)}{2} - r(S,t),\] where $\chi(S)$ is the Euler characteristic of $S$, $I(S) = |\partial S \cap s(\Gamma)|$ is the geometric intersection of $\partial S$ with $s(\Gamma)$, and $r(S,t)$ is the rotation of the projection into $v_0^\perp$ of the positive unit normal field of $S$ along $\partial S$ with respect to $t$.
\end{definition}

\begin{definition}\cite{MR2390347} \label{defn:boundcoh} A curve $C \subset R(\Gamma)$ is said to be {\em boundary-coherent} if either $[C] \neq 0$ in $H_1(R(\Gamma))$ or $C$ is oriented as the boundary of its interior.
\end{definition}

Now recall (Definition \ref{defn:HDforlink}) that one obtains a filtered chain complex, $X^{(0,1)}$, from a $(0,1)$ sutured multi-diagram, $(\Sigma, \boldalpha, \boldbeta_{\{0,1\}^\ell})$, for a framed $\ell$--component link, $L$, in a balanced sutured manifold.  If $S$ is a decomposing surface geometrically disjoint from $L$, then $X^{(0,1)}$ splits according to ${\bf A}_S$ gradings: \[X^{(0,1)} = \bigoplus_{m} X^{(0,1)}_m,\] where $X^{(0,1)}_m$ is the subcomplex of $X^{(0,1)}$ corresponding to those elements ${\bf x} \in CFH(Y(\cI))$ satisfying ${\bf A}_{\SI}({\bf x}) = m$.

\begin{theorem}  \label{thm:SurfDecomp} Let $L \subset (Y,\Gamma)$ be a framed link in a strongly-balanced 
 sutured manifold, and let $S \subset (Y,\Gamma)$ be a connected decomposing surface satisfying:
\begin{enumerate}
  \item $S \cap L = \emptyset$, and
  \item for every component $V$ of $R(\Gamma)$ the closed components of the intersection $V \cap S$ consist of parallel oriented boundary-coherent curves.
\end{enumerate}

Let $(Y', \Gamma')$ be the sutured manifold obtained by decomposing along $S$ and $L' \subset (Y',\Gamma')$ the induced image of $L$.  If $X^{(0,1)}$ (resp., $\left(X^{(0,1)}\right)'$) is the filtered complex associated to a $(0,1)$ sutured multi-diagram for $L$ (resp., $L'$), then $(X^{(0,1)})'$ is filtered quasi-isomorphic to $X^{(0,1)}_k$ for $k={\frac{1}{2}c(S,t)}$.


\end{theorem}

We have abused notation above, using $c(S,t)$ to refer to $c(\SI,t)$ for any choice of $\cI \in \{0,1\}^\ell$ (since $c(S_{\cI_1},t) = c(S_{\cI_2},t)$ for all pairs $\cI_1, \cI_2 \in \{0,1,\infty\}^\ell$).  The above theorem says that $(X^{(0,1)})'$ is filtered quasi-isomorphic to a direct summand of $X^{(0,1)}$ under the splitting induced by $S \subset (Y,\Gamma)$.  Furthermore, $c(S,t)$ identifies the Alexander grading of this direct summand.

\subsection{Surface multi-diagrams} 
To prove Theorem \ref{thm:SurfDecomp}, we will need special multi-diagrams compatible with a collection of decomposing surfaces.

\begin{definition} \label{defn:SurfMultDiag} Let \[\{(Y_j,\Gamma_j)|j=1, \ldots, n\}\] be a finite collection of balanced, sutured manifolds and $S_j \subset Y_j$ a decomposing surface for each $j \in \{1, \ldots, n\}$, in the sense of \cite[Defn. 2.4]{MR2390347}.  A balanced sutured multi-diagram adapted to the collection $\{S_j \subset Y_j|j = 1, \ldots, n\}$ is a tuple $(\Sigma, \boldalpha, \boldbeta_1, \ldots, \boldbeta_n, P)$ such that
\begin{enumerate}
  \item $(\Sigma, \boldalpha, \boldbeta_j)$ is a balanced sutured Heegaard diagram defining the sutured manifold $Y_j$ in the sense of \cite[Defn. 2.8]{MR2253454},
  \item $P \subset \Sigma$ is a closed subsurface of $\Sigma$ with $\partial P$ a graph satisfying:
    \begin{itemize}
       \item $\partial P = A \cup B$ such that $B \cap \boldalpha = \emptyset$ and $A \cap \boldbeta_j = \emptyset$ for all $j=1, \ldots, n$, and
       \item $A \cap B = P \cap \partial \Sigma$ is the set of vertices of the graph.
    \end{itemize}
  \item $(\Sigma, \boldalpha, \boldbeta_j, P)$ is adapted to $S_j$ in the sense of \cite[Defn. 4.3]{MR2390347}.  In particular, for each $j = 1, \ldots, n$, the surface obtained by smoothing the corners of \[\left(P \times \left\{\frac{1}{2}\right\}\right) \cup \left(A \times \left[\frac{1}{2},1\right]\right) \cup \left(B \times\left[0,\frac{1}{2}\right]\right) \subset Y_j\] is equivalent to $S_j$ (in the sense of \cite[Defn. 4.1]{MR2390347}).
\end{enumerate}
As in \cite[Defn. 4.3]{MR2390347}, we call a tuple $(\Sigma, \boldalpha, \boldbeta_1, \ldots, \boldbeta_n,P)$ satisfying the above properties a {\em surface multi-diagram}.
\end{definition}

\begin{definition} \label{defn:DecompSurf} Given a sutured multi-diagram $(\Sigma, \boldalpha, \boldbeta_1, \ldots, \boldbeta_n, P)$ adapted to a collection, $\{S_j \subset Y_j\}$, of decomposing surfaces, we can uniquely associate to it a tuple $(\Sigma', \boldalpha',\boldbeta_1', \ldots, \boldbeta_n', P_A, P_B, p)$, where $(\Sigma', \boldalpha', \boldbeta_1', \ldots, \boldbeta_n')$ is a balanced sutured multi-diagram such that for each $j = 1, \ldots, n$, the tuple $(\Sigma', \boldalpha', \boldbeta_j', P_A, P_B, p)$ represents the balanced sutured manifold obtained by decomposing $(Y_j,\Gamma_j)$ along $S_j$ and satisfies the properties listed in \cite[Defn. 5.1]{MR2390347}.

In particular, $P_A, P_B \subset \Sigma'$ are closed subsurfaces, $p: \Sigma' \rightarrow \Sigma$ is smooth, a local diffeomorphism on its interior, and its restrictions,
\begin{itemize} 
  \item $p: (\Sigma' - (\overline{P_A \cup P_B})) \rightarrow (\Sigma - \overline{P})$,
  \item $p:P_A \rightarrow P$, and
  \item $p:P_B \rightarrow P$
\end{itemize}
are diffeomorphisms.
\end{definition}


Suppose that $L=L_1 \amalg \ldots \amalg L_\ell$ is a framed link in a sutured manifold $(Y,\Gamma)$ and $(S,\partial S) \subset (Y,\partial Y)$ is a decomposing surface satisfying conditions (1) and (2) in the statement of Theorem \ref{thm:SurfDecomp}.  As in Section \ref{sec:LinkSurg}, for each $\cI \in \{0,1,\infty\}^\ell$ we let $\YI$ denote the sutured manifold obtained by doing $\cI$--surgery on $L$ and let $\SI$ denote the surface in $\YI$ compatible to $S \subset Y$ in the sense of Definition \ref{defn:CompSurf}.

Since surgery on $L \subset Y$ affects neither $\partial S \cap \Gamma$ nor $\partial S \cap R(\Gamma)$, $\SI$ is a decomposing surface satisfying conditions (1) and (2) in the statement of Theorem \ref{thm:SurfDecomp} for each $\cI \subset \{0,1,\infty\}^\ell$.

A decomposing surface in a sutured manifold, $(Y,\Gamma)$, determines a distinguished set of elements in Spin$^c(Y,\Gamma)$ as follows.  This set will play a large role in the proof of Theorem \ref{thm:SurfDecomp}.

\begin{definition} \label{defn:outer} \cite[Defn. 1.1]{MR2390347} A Spin$^c$ structure $\spincs \in Spin^c(Y,\Gamma)$ is said to be {\em outer} with respect to a decomposing surface $(S,\partial S) \subset (Y,\partial Y)$ if there exists a unit vector field on $Y$ representing $\spincs$ which is nowhere equal to the negative unit normal vector field on $S$.  As in \cite{MR2390347}, we denote the set of $S$--outer Spin$^c$ structures on $(Y,\Gamma)$ by $O_S$.  
\end{definition}


\begin{notation} \label{defn:Outer}
Given a compatible collection, $\{\SI \subset \YI\}$, of properly-imbedded, oriented surfaces as in Definition \ref{defn:CompSurf}, we denote by $\outerspinc$ the subset \[\coprod_{\cI} O_{\SI} \subset \coprod_{\cI} Spin^c(\YI)\] corresponding to the union over all $\SI$--outer Spin$^c$ structures.
\end{notation}



Given a link $L\subset (Y,\Gamma)$ and a decomposing surface $S \subset Y$ satisfying conditions (1) and (2) of the statement of Theorem \ref{thm:SurfDecomp}, the following two results explain how to construct a full sutured multi-diagram for $L$ that is also a surface multi-diagram adapted to the collection $\{\SI \subset \YI\,|\,\cI \in \{0,1,\infty\}^\ell\}$.

\begin{lemma}
Let $L$ be an oriented, framed link in a strongly-balanced sutured manifold $(Y,\Gamma)$ and let $S \subset (Y,\Gamma)$ a decomposing surface satisfying conditions (1) and (2) in the statement of Theorem \ref{thm:SurfDecomp}.   For each $\cI \in \{0,1,\infty\}^\ell$, let $(\YI',\Gamma')$ the sutured manifold obtained by decomposing $(\YI,\Gamma)$ along $\SI$.

Then $S$ is isotopic, in the complement of $L$, to a decomposing surface $S'$ such that each component of $\partial S'$ intersects both $R_+(\Gamma)$ and $R_-(\Gamma)$.  Furthermore, for each $\cI \in \{0,1,\infty\}^\ell$, decomposing $\YI$ along $\SI'$ also gives $(\YI',\Gamma')$, and $\outerspinc = \outerspinc'$.
\end{lemma}

\begin{proof}
The isotopy $S \rightarrow S'$ described in the proof of \cite[Lem. 4.5]{MR2390347} takes place in a neighborhood of $\partial Y$, hence avoids $L \subset \mbox{Int}(Y)$.  The rest of the result then follows from \cite[Lem. 4.5]{MR2390347}.
\end{proof}

\begin{proposition} \label{prop:Adaptedsurf} 
Let $L$ be an oriented, framed link in a strongly-balanced sutured manifold, and let $S$ be a decomposing surface satisfying:
\begin{enumerate}
\item $S \cap L = \emptyset$, and
\item each component of $\partial S$ intersects both $R_+(\Gamma)$ and $R_-(\Gamma)$.
\end{enumerate}

Then there exists an admissible full sutured multi-diagram for $L$ that is also a surface multi-diagram adapted to the collection $\left\{\SI \subset \YI | \cI \in \{0,1,\infty\}^\ell\right\}$.
\end{proposition}

See \cite[Sec. 3.2]{GT08071432} for the definition of an admissible sutured multi-diagram.

\begin{proof} Since each component of $\partial S$ intersects both $R_\pm(\Gamma)$, it follows that the boundary of each component of $S$ intersects both $R_\pm(\Gamma)$ and there are no closed components of $\partial S$ contained entirely in $\Gamma$.  

We begin as \Juhasz does in the proof of \cite[Prop. 4.4]{MR2390347}, 
by defining a function \[f: \partial Y \cup S \rightarrow [-1,4]\] with $f^{-1}\left(\frac{3}{2}\right) \cap S$ a polygon, $P \subset S$.  Since $S \cap L = \emptyset$, $L$ is contained in the complement, $Y - N(S)$.  Furthermore, since the boundary of each component of $S$ intersects $R_+(\Gamma)$, each component of $Y-N(S)$ intersects $R_+(\Gamma)$. 

We now form a bouquet on $L$ as in \cite{GT08071432}[Sec. 4] by choosing a collection of arcs $a_1, \ldots, a_\ell$ from the link components $L_1, \ldots L_\ell$ to $R_+(\Gamma)$.  Let $L':= N(L \cup a_1 \cup \ldots \cup a_\ell)$ as in the discussion preceding Definition \ref{defn:CompSurf}, and extend the function, $f$, to a Morse function, $f_0$, on $Y - L'$.  We deform this Morse function to a self-indexing one with no index $0$ or index $3$ critical points.  In the process, we may need to move $S$ to an isotopic surface, $S'$.  

Let $\Sigma = f^{-1}\left(\frac{3}{2}\right)$ and let $\boldalpha = (\alpha_1, \ldots, \alpha_d)$ (resp., $\boldbeta_0 = (\beta_{\ell+1}, \ldots, \beta_d)$) be the intersection of gradient flow lines from the index $1$ critical points (resp., to the index $2$ critical points).  Then $\left(\Sigma, \boldalpha, \boldbeta_0\right)$ is a (non-balanced) sutured Heegaard diagram for $Y-L'$ compatible with $S'$.  Each $\mathcal{I} \in \{0,1,\infty\}^\ell$ specifies an extension of $\boldbeta_0$ to $\boldbeta_\cI$ as in \cite[Sec. 4]{GT08071432}.  $\left(\Sigma,\boldalpha,\boldbeta_{\{0,1,\infty\}^\ell}\right)$ is then a full sutured multi-diagram for $L$ that is also a surface multi-diagram adapted to the collection, $S_\cI \subset Y(\cI)$.  In particular, $P \subset \Sigma$ is a closed subsurface with $\partial P$ a graph satisfying: 
\begin{itemize}
  \item $\partial P = A \cup B$ such that $B \cap \boldalpha = \emptyset$ and $A \cap \boldbeta_{\cI} = \emptyset$ for all $\cI \in \{0,1,\infty\}^\ell$, and
  \item $A \cap B = P \cap \partial \Sigma$ is the set of vertices of the graph.\end{itemize} 

To make $\left(\Sigma, \boldalpha, \boldbeta_{\{0,1,\infty\}^\ell}\right)$ admissible, wind the curves in $\boldbeta_{\{0,1,\infty\}^\ell}$ in the complement of $A \subset P$, as in \cite[Prop. 4.8]{MR2390347}, employing the inductive argument from the proof of \cite[Lem. 3.12]{GT08071432}.
\end{proof}

Note that Proposition \ref{prop:Adaptedsurf} implies the existence of an admissible surface multi-diagram for $L$ adapted to $\{\SI \subset Y(\cI)\,|\,\cI \in \{\cI_{i_1}, \ldots, \cI_{i_k}\}\}$ for any subset, $\{\cI_{i_1}, \ldots, \cI_{i_k}\} \subset \{0,1,\infty\}^\ell.$





\subsection{Polygon counts in cylindrical sutured Floer homology} \label{sec:polycounts}

In \cite{MR2240908}, Lipshitz provides a reformulation of the Heegaard Floer homology of a closed $3$--manifold, $Y$, using counts of holomorphic curves in the symplectic manifold $\Sigma \times [0,1] \times \R$, where $\Sigma$ is a Heegaard surface for $Y$.  Sutured Floer homology admits an analogous reformulation, as briefly described in \cite[Sec. 2.1]{GT08071432}.  In this section, we discuss the cylindrical analogue of holomorphic polygon counts in sutured multi-diagrams.  This construction is implicit in \cite[Sec. 10]{MR2240908}, which describes holomorphic triangle and $4$--gon counts.

We will also make use of some of the theory describing moduli space degenerations appearing in \cite{GT08100687}.  In order to place ourselves in the correct context to directly apply relevant results, we will sometimes find it convenient to fill in each boundary component of $\Sigma$ with a disk containing a basepoint, obtaining a closed Heegaard surface, $\hatSigma$.  This poses no extra complications, since counts of holomorphic polygons whose domains miss the basepoint regions correspond naturally to holomorphic polygon counts in the original sutured setting.



\begin{definition} \label{defn:ngon} By a {\em holomorphic $n$--gon} ($n \geq 2 \in \Z$) we mean any region in $\C$ with smooth boundary and $n$ cylindrical ends, as pictured in Figure \ref{fig:Pentagon}.\footnote{Any such region is conformally equivalent to the unit disk in $\C$ with $n$ boundary punctures.}  When $n\leq 3$, all such objects are conformally equivalent,\footnote{When $n=2$, the space of conformal equivalences is homeomorphic to $\R$.}  and when $n>3$, the moduli space of (marking-preserving conformal equivalence classes of) such objects is parameterized by a space homeomorphic to $\R^{n-3}$.\footnote{This parameterization is not natural, but can be accomplished by, for example, using a conformal equivalence to identify three of the boundary punctures with three specified, fixed points on the unit circle and parameterizing by arclength from the variable punctures to a fixed puncture.}   We will use $\PP_n$ to denote an $n$--gon, considered as a topological object, and $(\PP_n,j_\PP)$ to denote an $n$--gon equipped with a specific complex structure, $j_{\PP} \in \R^{n-3}$.  

As in \cite[Fig. 7]{MR2240908}, each $\PP_n$ is equipped with a clockwise labeling of its boundary components: $e_0, \ldots, e_{n-1}$, and an associated labeling of its cylindrical ends: $v_{0,1}, \ldots, v_{n,0}$.  We denote by $\cM_{\PP_n}$ the moduli space of marked conformal equivalence classes of $n$--gons.
\end{definition}


\begin{figure}
\begin{center}
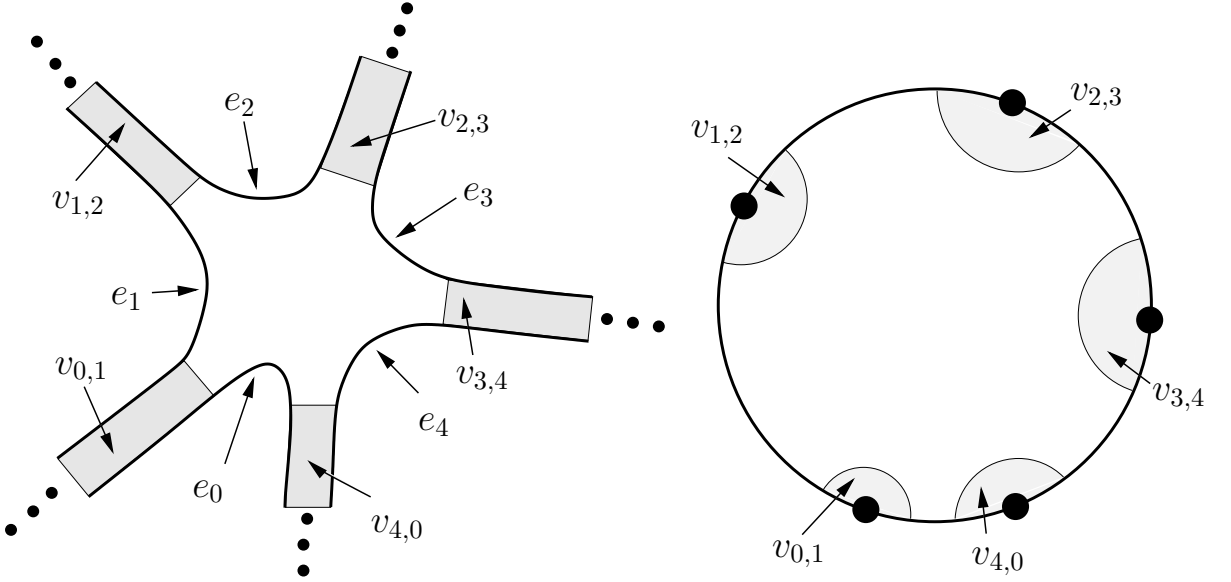
\end{center}
\caption{A holomorphic $5$--gon $\subset \C$ (left) with edges labeled $e_0, \ldots, e_4$ and cylindrical ends labeled $v_{0,1}, \ldots, v_{4,0}$, conformally equivalent to the unit disk $\subset \C$ with $5$ boundary punctures (right).}
\label{fig:Pentagon}
\end{figure}

\vskip 20pt
Now let $(\Sigma, \boldeta_0, \ldots, \boldeta_n)$ be a sutured multi-diagram, where, for each $i = 0, \ldots, n$, \[\boldeta_i = (\eta_i)_1 \amalg \ldots \amalg (\eta_i)_d\] is a disjoint union of $d$ circles ($d \geq 0$ a fixed integer) imbedded in $\Sigma$, such that the set, $\{[(\eta_i)_1], \ldots, [(\eta_i)_d]\},$ is linearly-independent in $H_1(\Sigma;\Z)$.

We form the $4$--manifold \[\Weta := \Sigma \times \PP_{n+1},\] 
equipped with the obvious projection maps, $\pi_\Sigma: \Weta \rightarrow \Sigma$ and $\pi_\PP: W \rightarrow \PP_{n+1}$, and let \[W_{\eta_i, \eta_{i+1}}= \Sigma \times v_{i,i+1}\] denote the cylindrical ends of $\Weta$.  
We now endow $\Weta$ with a split symplectic form, $\omega = \omega_{\Sigma} + \omega_{\PP}$ and choose generic almost complex structures on the cylindrical ends satisfying \cite[({\bf J1})-({\bf J5})]{MR2240908} along with a family, \[\{J_a\,|\,a \in \cM_{\PP_{n+1}}\},\] of almost complex structures on $\Weta$ satisfying \cite[({\bf J'1})-({\bf J'4})]{MR2240908} when $n+1 = 3$ and the obvious analogues of conditions $(1)$-$(7)$ of \cite[Sec. 10.6.2]{MR2240908} when $n+1 \geq 4$.\footnote{Conditions $(5)$ and $(6)$ are made recursively with reference to all possible degenerations of an $n$--gon as in \cite[Sec. 8.1.5]{MR2113019}, \cite[Fig. 3]{GT08071432}.}  We will say that a family of almost complex structures as above {\em satisfies Lipshitz's conditions.}

An intersection point, ${\bf x} \in \boldeta_i \cap \boldeta_{i+1}$, is now defined to be any $d$--tuple of distinct points, \[{\bf x} = (x_1, \ldots, x_d) \in \boldeta_i \cap \boldeta_{i+1},\] such that exactly one $x_j$ lies on each circle of $\boldeta_i$ and exactly one $x_j$ lies on each circle of $\boldeta_{i+1}$.  

For each $i = 0, \ldots, n$, let $C_{\boldeta_i}$ denote the disjoint union of cylinders, $\boldeta_i \times e_i \subset \Weta$.  

Now, given \[({\bf x}_0, \ldots, {\bf x}_n) \in (\boldeta_0 \cap \boldeta_1) \times \ldots \times (\boldeta_n \cap \boldeta_0),\] let $\pi_2({\bf x}_0, \ldots, {\bf x}_n)$ denote the set of homology classes of continuous maps \[f: (\cR, \partial \cR) \rightarrow \left(\Weta, \bigcup_{i=0}^n C_{\boldeta_i}\right)\] such that
\begin{itemize}
  \item $\cR$ is a (non-compact) surface with boundary and cylindrical ends, and
  \item the cylindrical ends of $\cR$ are asymptotic to ${\bf x}_i \times v_{i,i+1}$ in a one-to-one fashion.  In particular, there is a bijective correspondence between cylindrical ends of $\cR$ and cylinders, $(x_i)_j \times v_{i,i+1}$.
\end{itemize}

If $\cA \in \pi_2({\bf x}_0, \ldots, {\bf x}_n)$ is a homology class as above (see \cite[Sec. 10]{MR2240908}), then it corresponds naturally to a homotopy class $\phi_{\cA} \in \pi_2({\bf x}_0, \ldots, {\bf x}_n)$ (in the sense of \cite[Sec. 8.1.2]{MR2113019}).  

Furthermore, let $\left(\Weta\right)_{J_a}$ represent the $4$--manifold, $\Weta$, equipped with the complex structure, $J_a$.  Then if we denote by $\cM^\cA_{J_a}$ the moduli space of embedded holomorphic curves \[f: (\cR, \partial \cR) \rightarrow \left(\left(\Weta\right)_{J_a}, \bigcup_{i=0}^n C_{\boldeta_i}\right),\] and by \[\cM^{\cA}:= \bigcup_{a \in \cM_{\PP_{n+1}}} \cM^{\cA}_{J_a}\] the union over the moduli spaces associated to the marked equivalence classes, then there is a ``tautological correspondence''\footnote{with respect to suitable almost complex structures} (\cite{MR1978044}, \cite[Sec. 13]{MR2240908}) between  $\cM^A $ 
and $\cM(\phi_A)$, the moduli space of holomorphic $n+1$--gons in $Sym^d(\Sigma)$ representing $\phi_A$ (with respect to suitable almost complex structures).
\subsection{Proof of surface decomposition theorem for sutured multi-diagrams}

With the relevant background material in place, we turn to the proof of Theorem \ref{thm:SurfDecomp}.  The following technical lemma is crucial to allowing an identification of the appropriate moduli spaces appearing in the definition of the boundary map for the chain complexes in Theorem \ref{thm:SurfDecomp}.  The idea for the proof was suggested to us by Robert Lipshitz.

\begin{lemma} \label{lem:DomainSplit}Let $(\Sigma,\boldeta_0, \boldeta_1, \ldots, \boldeta_n,P)$ be an admissible surface multi-diagram (Definition \ref{defn:SurfMultDiag}).\footnote{In Definition \ref{defn:SurfMultDiag}, we used the notations $\boldalpha$ instead of $\boldeta_0$ and $\boldbeta_i$ instead of $\boldeta_i$, for $1\leq i\leq n$.}  Recall that $P \subset \Sigma$ is an imbedded subsurface with $\partial P = A \cup B$, such that 

\begin{tabular}{ll}
  $\boldeta_i \cap B = \emptyset$ & if $i=0$, and\\
  $\boldeta_i \cap A = \emptyset$ & if $i\neq 0$.
\end{tabular}

Let ${\bf x}_i \in {\boldeta_i} \cap {\boldeta_{i+1}}$ for each $i \in \Z_{n+1}$, and suppose that ${\bf x}_0, {\bf x}_n$ are {\em outer} in the sense of \cite[Defn. 5.3]{MR2390347}.  

Fix $\cA \in \pi_2({\bf x}_0, {\bf x}_1, \ldots, {\bf x}_n)$ such that $\mu(\phi_{\cA}) =0$ (resp., $\mu(\phi_{\cA}) = 1$) if $n>1$ (resp., if $n=1$).\footnote{We make this index restriction because the proof of Theorem \ref{thm:SurfDecomp} only requires an identification of 0-dimensional moduli spaces of polygons.  Furthermore, Lipshitz imposes transversality requirements on families of almost complex structures \cite[Sec. 10.6.2, (1)--(7)]{MR2240908} only for curves of index $\leq 1$.} 

Then there exists a family, $\{J_a\,|\,a \in \cM_{\PP_{n+1}}\}$, of almost complex structures on $\Weta$ satisfying Lipshitz's conditions 
such that for every $(\cR,\partial \cR) \in \cM^\cA$, \[\cR_P := \cR \cap \pi_\Sigma^{-1}(P)\] splits as a disjoint union \[\cR_P = \cR_A \amalg \cR_B,\] where 
$\phi_A := \pi_\Sigma(\cR_A)$ and $\phi_B:= \pi_\Sigma(\cR_B)$ satisfy
\begin{eqnarray*}
\partial\phi_A &\subset& (A \cup \boldeta_0)\\
\partial\phi_B &\subset& (B \cup_{i\neq 0} \boldeta_i).
\end{eqnarray*}
\end{lemma}

We summarize Lemma \ref{lem:DomainSplit} as follows: There exists a family of almost complex structures on $\Weta$ with respect to which the intersection of a fixed Maslov index $0$ holomorphic curve with $\pi_{\Sigma}^{-1}(P)$ splits as the disjoint union of a curve having boundary {\em of type $A$} and a curve having boundary {\em of type $B$}.

\begin{proof} 
Let $a_1, \ldots, a_r$ (resp., $b_1, \ldots, b_r$) be the edges in $A \subset \partial P$ (resp., $B \subset \partial P$).  
Form $\hatSigma$ by filling in the boundary components of $\Sigma$ with disks containing basepoints.\footnote{See the discussion at the beginning of Section \ref{sec:polycounts}.}
Denote by $\hatP \subset \hatSigma$ the image of $P$, and by $Z_1,\ldots,Z_m \subset \hatSigma$ the images of the boundary components of $P$, under the inclusion map.

Given an almost complex structure, $j$, on $\hatSigma$, let $\ell(Z)$ denote the length of a closed curve, $Z$, with respect to a fixed complete hyperbolic metric associated to $j$ (c.f.\cite[Chp. IV]{MR1451624}).

Choose a generic sequence, $\left\{j_{\Sigma}\right\}_{i \in \N}$, of almost complex structures on $\hatSigma$ with respect to which $\ell(Z_1), \ldots, \ell(Z_m) \rightarrow 0$ as $i \rightarrow \infty$ and a corresponding sequence, $\left\{J_a|a \in \cM_{\PP_{n+1}}\right\}_{i\in \N}$ , of families of almost complex structures on $\Weta$ satisfying Lipshitz's conditions.


Suppose $\cM^\cA \neq \emptyset$ for infinitely many choices of $(j_\Sigma)_i$,  $i \in \N$.  Then there exists a subsequence of imbedded holomorphic surfaces, \[(\cR_{i_k}, \partial \cR_{i_k}) \rightarrow W_{\eta_0, \ldots, \eta_n},\] such that
\begin{itemize}
  \item each $\cR_{i_k}$ represents the homology class $\cA$,
  \item the associated domain, $\phi(\cA) = \piSigma(\cR_{i_k})$, of each $\cR_{i_k}$ misses the basepoints, and 
  \item the sequence converges to a matched pair of {\em simple holomorphic combs} $\cR_{P}$ and $\cR_{\Sigma - P}$,  where  $\cR_{P}$ (resp., $\cR_{\Sigma - P}$) projects to $\widehat{P}$ (resp., to $\widehat{\Sigma} - \widehat{P}$) under the map $\pi_\Sigma$.  We say that $\cR_{P}$ (resp., $\cR_{\Sigma - P}$) is the holomorphic comb over $\widehat{P}$ (resp., over $\widehat{\Sigma} - \widehat{P}$).
\end{itemize}

Here we adopt the terminology from \cite[Sec. 5, 9]{GT08100687}.  By a straightforward analogue of the arguments found in \cite[Prop. 5.20]{GT08100687} and \cite[Prop. 4.2.1]{Lipshitzthesis}, matched pairs of simple holomorphic combs form the top stratum of the compactification of $\cM^{\cA}$.

Recall that a simple holomorphic comb \cite[Defn. 5.15]{GT08100687} is a pair, $(u,v)$, of holomorphic imbeddings of surfaces with decorated cylindrical ends, where $u$ maps to $\Sigma \times \PP_{n+1}$, $v$ maps to $\circlesZ \times \R \times \PP_{n+1},$ and there is a $1$ to $1$ correspondence between the Reeb decorations of $u$ at east $\infty$ and the Reeb chord decorations of $v$ at west $\infty$.    See \cite[Sec. 5.2, Conditions (1)-(5)]{GT08100687} for the appropriate conditions on the curve $v$ in $\circlesZ \times \R \times \PP_{n+1}$ in the case $n=1$.  When $n>1$, we replace condition (3) with the requirement that $\partial v \subset \partial \PP_{n+1}$.  This condition, combined with the maximum principle,  will force $\PP_{n+1}$ to be equipped with a degenerate holomorphic structure, and, under the map $\pi_\PP$, $v$ will map to a single point in (the degeneration of) $\PP_{n+1}$.

More explicitly, since $\piP: \cR_{i_k} \rightarrow \PP_{n+1}$ is holomorphic for each $i_k$, (\cite[Condition {\bf J'4}]{MR2240908}), the projections, $\piP$, of the interiors of the limiting curves, $\cR_P$ and $\cR_{\Sigma - P}$, must also be holomorphic.  Hence, any cylindrical end of $\cR_P$ (resp., $\cR_{\Sigma - P}$) must project to a cylindrical end of $\PP_{n+1}$.  In other words, $\PP_{n+1}$ is equipped with a (degenerate) holomorphic structure such that the images of the cylindrical ends of $\cR_P$ (resp., $\cR_{\Sigma - P}$) project to cylindrical ends of (the degeneration of) $\PP_{n+1}$.  Each cylindrical end can be topologically identified with some open half-neighborhood of a smoothly-imbedded arc in $\PP_{n+1}$.  We will refer to these arcs as {\em vanishing arcs} and label them $\gamma_1, \ldots, \gamma_r$.  See Figure \ref{fig:Polygon}.

\begin{figure}
\begin{center}
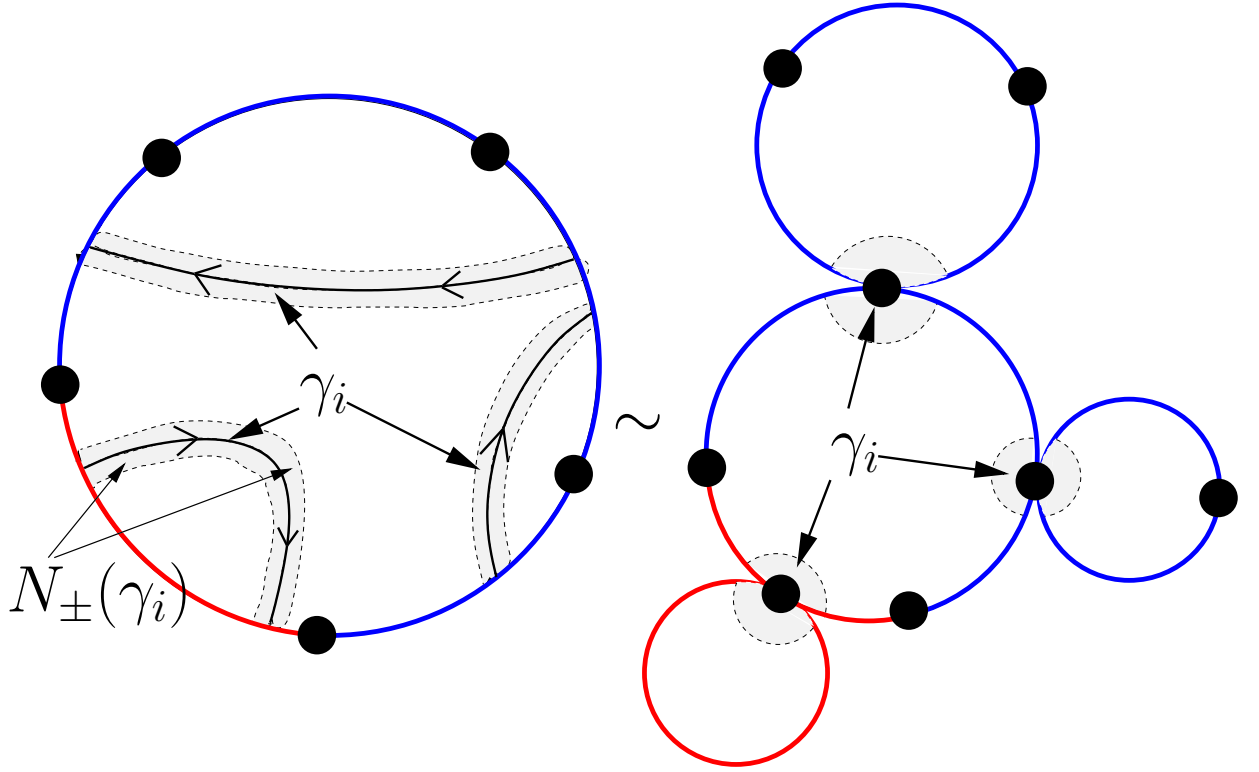
\end{center}
\caption{A polygon, $\PP_{n+1}$ (left), pictured as a unit disk with boundary punctures, along with the oriented vanishing arcs, $\{\gamma_i\}$, induced by the degeneration of the holomorphic structure on $\Sigma$ along $Z_1, \ldots, Z_m$.  The half-neighborhoods, $N_\pm(\gamma_i)$, of the smoothly imbedded arcs, $\gamma_i$, are equipped with cylindrical holomorphic structures, as indicated in the figure on the right.  (Note that we may need to rescale to deal with flat components as in, e.g., \cite[Proof of Lemma 8.2]{MR2240908}.)}
\label{fig:Polygon}
\end{figure}

Each connected component of $\PP_{n+1}- \cup_{i}\gamma_i$ is then conformally equivalent to some holomorphic $k$--gon ($1 \leq k \leq n+1$) which, upon fixing a parameterization (see Definition \ref{defn:ngon}), can be conformally identified with the standard unit disk with marked points on the boundary.\footnote{Neighborhoods of the marked points correspond to cylindrical ends under this identification.}  Each holomorphic curve, \[v \subset \circlesZ \times \R \times \PP_{n+1},\] is mapped, under $\pi_\PP$, to the image of one of the vanishing arcs under this identification.





A pair, $(u_1,v_1)$ and $(u_2,v_2)$ of holomorphic combs is {\em matched} if the asymptotics of $v_1$ at east $\infty$ match the (orientation-reverse) of the asymptotics of $v_2$ at east $\infty$ \cite[Sec. 9]{GT08100687}.    More precisely, let $\PP_{n+1}$ be equipped with a (possibly degenerate) holomorphic structure as above, with vanishing arcs $\gamma_1, \ldots, \gamma_r$.  Choose orientations on the $\gamma_i$ and label by $N_+(\gamma_i)$ (resp., $N_-(\gamma_i)$) the open half-neighborhood of $\gamma_i$ oriented compatibly (resp., incompatibly) with $\gamma_i$.\footnote{\label{footnote:orientarcs}Note that an orientation on the circles, $Z_i$, will induce an orientation on the vanishing arcs, $\gamma_i$.}  Then we say that $(u_1,v_1)$ and $(u_2,v_2)$ are a matched pair of holomorphic combs if, under the map, $\pi_\PP$, the interior of $u_i$ maps holomorphically to the interior of \[\PP_{n+1} - \cup_i \gamma_i\] in such a way that 

\begin{itemize}
  \item there is a 1 to 1 (cyclic-orientation-reversing on each boundary component) identification of the limiting Reeb chords of $v_1$ at east $\infty$ with the limiting Reeb chords of $v_2$ at east $\infty$, and
  \item if a cylindrical end of $u_1$ projects to $N_+(\gamma_i)$ under the map $\pi_\PP$, and it is identified (under the sequence of 1 to 1 correspondences described above and in the definition of a holomorphic comb) with a cylindrical end of $u_2$, then the corresponding cylindrical end of $u_2$ maps to $N_-(\gamma_i)$.
  \item If $n=1$ (hence $\PP_{n+1} \sim [0,1] \times \R$), the evaluation maps $ev: u_i \rightarrow \R$ (as described in \cite[Sec. 5.1]{GT08100687}) agree under the correspondence.
\end{itemize}

\begin{remark} Since our chosen sequence, $\left\{j_\Sigma\right\}_{i \in \N}$, of almost complex structures on $\Sigma$ is generic, and we are restricting our attention to index $0$ (resp., index $1$) domains when $n > 1$ (resp, $n=1$), we may assume, for any comb, $(u,v)$, we encounter, that the part, $v$, mapping into $\circlesZ \times \R \times \PP_{n+1}$ is {\em trivial}, in the sense of the discussion following \cite[Defn. 5.14]{GT08100687}.

Bearing this in mind, whenever we refer to a holomorphic comb, we will henceforth mean a holomorphic imbedding, $u$, into $\Sigma \times \PP_{n + 1}$ of a (possibly disconnected) surface with decorated cylindrical ends limiting on Reeb chords in $\circlesZ \times \PP_{n+1}$.
\end{remark}

Since we are interested in the topology of $\phi_{P}$, let us focus on the holomorphic comb over $\hatP$.  More specifically, for each $i_k$, let $(\cR_{P})_{i_k}$ represent $\cR_{i_k} \cap \piSigma^{-1}(\hatP)$, 
and let $\cR_P$ represent the holomorphic comb which projects, via $\pi_\Sigma$, to $\hatP$ in the limit.



Then $\cR_P$ has cylindrical ends of two types:

\begin{itemize}
  \item those asymptotic to ${\bf x}_i \times v_{i,i+1}$ (these are the analogues of Reeb chords at $\pm \infty$, in the language of \cite[Sec. 5]{GT08100687}), and
  \item those asymptotic to $\rho$, where $\rho$ is a Reeb chord in $\circlesZ \times \PP_{n+1}$.
\end{itemize}

We now make two key observations:
\begin{enumerate}
  \item Each cylindrical end of $\cR_P$ limits on a Reeb chord whose boundary points are either both of type $A$ or both of type $B$.  Recall that a type $A$ (resp., type $B$) boundary is one that lies on curves in $\boldeta_0$ (resp., $\boldeta_i$ for $i \neq 0$).  This follows because the cylindrical ends of $\cR_P$ are asymptotic to Reeb chords in either ${\bf x}_i \times v_{i,i+1}$ or in $\circlesZ \times \PP_{n+1}$.  All cylindrical ends of the first type have boundaries of type $B$ only since, by assumption, ${\bf x}_0$ and ${\bf x}_n$--the only intersection points corresponding to Reeb chords with boundaries of both type $A$ and $B$--lie in outer Spin$^c$ structures, hence do not appear among the cylindrical ends of $\cR_P$.  By the positioning of the basepoints in the regions adjacent to the circles, $Z_i$ (see Figure \ref{fig:Reebchords}), each cylindrical end of the second type will have boundary points which are either both of type $A$ or both of type $B$.
  \item Each of \[\piP: \cR_{i_k} \rightarrow \PP_{n+1}\] along with the limit \[\piP: \cR \rightarrow \PP_{n+1}\] is a $d$--fold branched covering map \cite{MR2240908,GT08100687}.  
\end{enumerate}

\begin{figure}
\begin{center}
\input{Figures/Reebchords.pstex_t}
\end{center}
\caption{A subsurface $\widehat{P} \subset \widehat{\Sigma}$ associated to a surface multi-diagram.  The positioning of the basepoints in the regions adjacent to the boundary, $\circlesZ$, of $P$ forces the east $\infty$ cylindrical ends of holomorphic combs to have boundary of either type $A$ or type $B$ but not both.}
\label{fig:Reebchords}
\end{figure}
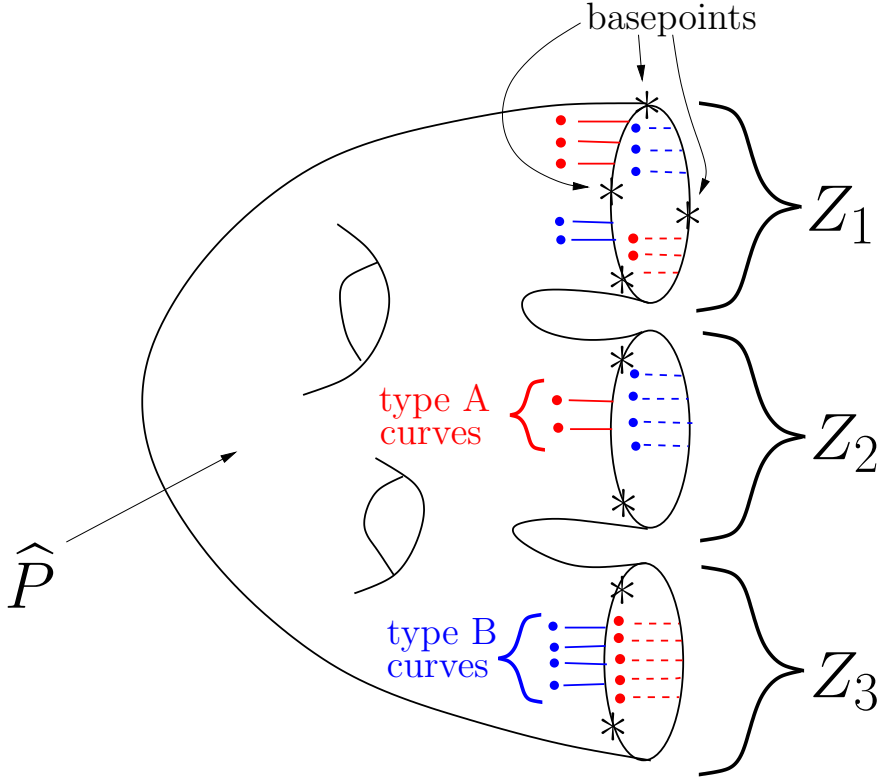

Orient the circles, $Z_1, \ldots, Z_m$, as the boundary of $\widehat{P}$, and let $\gamma_i$ be the associated oriented vanishing arcs of $\PP_{n+1}$ with respect to the sequence $\left\{\left(j_\Sigma\right)\,|\,i\in \N\right\}$ (see Footnote \ref{footnote:orientarcs}).  Then each east $\infty$ cylindrical end of $\cR_P$ is mapped to a particular element of $\{N_+(\gamma_i)\}$ under the projection, $\pi_\PP$.  

Now let $\{\widetilde{N}_+(\gamma_i)\}$ denote the set of connected components of the $d$--fold branched covers in $\cR$ of $\{N_+(\gamma_i)\}$.  Similarly, let $\{C_i\}$ denote the set of connected components of \[\PP_{n+1} - \bigcup \overline{N_+(\gamma_i)},\] and let $\{\widetilde{C}_i\}$ denote the connected components of their $d$--fold branched covers in $\cR$. By observation $1$, above, if $\widetilde{C}_i \subset \cR_P$, then $\widetilde{C}_i$ has curves of only one type (either type $A$ or type $B$) on its boundary.  Hence, the (necessarily connected) projection of $\widetilde{C}_i$ to $\Sigma$ has the same property.

Furthermore, each $\widetilde{N}_+(\gamma_i)$ has boundary curves of only one type.  Let 
\begin{eqnarray*}
(\widetilde{C}_P)_A &=& \bigcup \left\{\widetilde{C}_i \subset \cR_P\,\,|\,\,\partial\widetilde{C}_i \mbox{ is type A}.\right\}\\
(\widetilde{N}_+(\gamma))_A &=& \bigcup \left\{\widetilde{N}_+(\gamma_i)\,\,|\,\,\partial(\widetilde{N}_+(\gamma_i)) \mbox{ is type A}.\right\}
\end{eqnarray*}

Analogously define $(\widetilde{C}_P)_B$ and $(\widetilde{N}_+(\gamma))_B$.


Then \[\left( \left(\widetilde{C}_P\right)_A \cup \left(\widetilde{N}_+(\gamma)\right)_A\right) \cap \left(\left(\widetilde{C}_P\right)_B \cup \left(\widetilde{N}_+(\gamma)\right)_B \right) = \emptyset.\]

Taking 
\begin{eqnarray*}
  \cR_A &:=& \left( \left(\widetilde{C}_P\right)_A \cup \left(\widetilde{N}_+(\gamma)\right)_A\right), \mbox{ and}\\
  \cR_B &:=& \left(\left(\widetilde{C}_P\right)_B \cup \left(\widetilde{N}_+(\gamma)\right)_B \right),
\end{eqnarray*} 

we therefore can realize $\cR_P$ as a disjoint union: \[\cR_P = \cR_A \amalg \cR_B,\] such that the projection, $\pi_{\Sigma}(\cR_P) = \phi_{P}$, can be decomposed as \[\phi_{P} = \phi_A + \phi_B,\] where $\phi_A$ (resp., $\phi_B$) has boundary of type $A$ (resp., type $B$), as desired.

\end{proof}

Lemma \ref{lem:DomainSplit} implies that we can choose a suitable almost complex structure with respect to which a holomorphic polygon represented by any {\em fixed} domain, $\mathcal{A}$ (representing a topological polygon between generators in outer spin$^c$ structures), must split as claimed.  If the surface multi-diagram is admissible, we can find a suitable almost complex structure with respect to which holomorphic polygons for {\em all} such domains split as claimed:

\begin{corollary} \label{cor:DomainSplit} Let $(\Sigma, \boldeta_0, \ldots, \boldeta_n,P)$ be an admissible surface multi-diagram.  Then there exists some generic almost complex structure, $j_\Sigma$, on $\Sigma$ such that for {\em every} choice of the following data:

\begin{itemize}
  \item a non-empty ordered subset, \[\{i_0=0, i_1, \ldots, i_k\} \subseteq \{0, \ldots, n\}\] 
  \item an associated family, $\{J_a|a \in \cM_{\PP_{k+1}}\}$, of almost complex structures on $\Wetak$ satisfying Lipshitz's conditions,
  \item a $(k+1)$--tuple $({\bf x}_0, \ldots, {\bf x}_k)$ with ${\bf x}_j \in \mathbb{T}_{\eta_{i_j}} \cap \mathbb{T}_{\eta_{i_{j+1}}}$ and ${\bf x}_0, {\bf x}_k$ outer, and 
  \item a domain $\cA \in \pi_2({\bf x}_0, \ldots, {\bf x}_k)$,
\end{itemize}


each surface $(\cR,\partial \cR) \in \cM^\cA$ satisfies the property that \[\cR_P := \cR \cap \pi_\Sigma^{-1}(P)\] splits as a disjoint union \[\cR_P = \cR_A \amalg \cR_B,\] where 
$\phi_A := \pi_\Sigma(\cR_A)$ (resp., $\phi_B:= \pi_\Sigma(\cR_B)$) has boundary of type $A$ only (resp., type $B$ only).
\end{corollary}




\begin{proof} Since $(\Sigma, \boldeta_0, \ldots, \boldeta_n)$ is admissible, there are a finite number of domains that can represent holomorphic $(k+1)$--gons, for any $1 \leq k \leq n$, since there are a finite number of domains with non-negative coefficients.  Of these domains, let $\phi_1, \ldots, \phi_m$ be those for which $\phi_P := \phi_i \cap P$ does not split as $\phi_A + \phi_B$.  Now choose a sequence $\left\{j_\Sigma\right\}_{i \in \N}$, of almost complex structures on $\Sigma$ as in the proof of Lemma \ref{lem:DomainSplit} along with a choice of family, \[\{J_a | a \in \cM_{\PP_{k+1}}\},\] of almost complex structures on $\Wetak$ for each $1 \leq k \leq n$.  If, for each $(j_\Sigma)_i$, there exists some $\phi_l, \mbox{ with } l \in \{1, \ldots, m\},$ such that  $\cM(\phi_l) \neq \emptyset$, then, by the pigeonhole principle, there exists some fixed $l_0 \in \{1,\ldots, m\}$ and some subsequence $\left\{(j_\Sigma)_{i_{c}}\right\}_{c\in \N}$ such that $\cM(\phi_{l_0}) \neq \emptyset$ for all $(j_\Sigma)_{i_{k}}$, contradicting Lemma \ref{lem:DomainSplit}.

\end{proof}

\begin{proposition} \label{prop:ChainHomEquiv} Let $(\Sigma, \boldalpha, \boldbeta_{\{0,1\}^\ell}, P)$ be a $(0,1)$ sutured multi-diagram for $L \subset Y$ adapted to a collection, \[\{\SI \subset \YI | \cI \in \{0,1\}^\ell\},\] of decomposing surfaces $\SI$ for $\YI$, and let $L'$ be the image of $L$ in $(Y',\Gamma')$, the sutured manifold obtained by decomposing along $S$.  

Let $(\Sigma', \boldalpha', \boldbeta_{\{0,1\}^\ell}', P_A, P_B, p)$ be the tuple obtained from this data as in Definition \ref{defn:DecompSurf}.  Then $(\Sigma', \boldalpha', \boldbeta_{\{0,1\}^\ell}')$ is a $(0,1)$ sutured multi-diagram for $L' \subset Y'$.

Furthermore, \[X' \cong \bigoplus_{\spincs \in \outerspinc} (X;\spincs),\] where  ``$\cong$'' denotes a filtered quasi-isomorphism.

\end{proposition}

In the above, $X$ (resp., $X'$) is the filtered chain complex obtained from the $(0,1)$ sutured multi-diagram, $\left(\Sigma, \boldalpha, \boldbeta_{\{0,1\}^\ell}\right)$ $\left( \mbox{resp.,} \left(\Sigma', \boldalpha', \boldbeta'_{\{0,1\}^\ell}\right)\right)$, as in the discussion following Definition \ref{defn:HDforlink}, and $\bigoplus_{\spincs \in \outerspinc} (X;\spincs)$ refers to the complex whose generators lie in $\outerspinc$, the outer spin$^c$ structures, with restricted differential.

\begin{proof}
$(\Sigma',\boldalpha',\boldbeta'_{\cI})$ is a sutured Heegaard diagram for $Y'(\cI)$ for each $\cI$, by \cite[Prop. 5.2]{MR2390347}, since $(\Sigma,\boldalpha, \boldbeta_{\cI}, P)$ is adapted to $\SI \subset \YI$ for each $\cI \in \{0,1\}^\ell$.  Hence, \[\left(\Sigma', \boldalpha', \boldbeta'_{\{0,1\}^\ell}\right)\] is a $(0,1)$ sutured multi-diagram for $L' \subset Y'$.

In order to prove that \[X' \cong \bigoplus_{\spincs \in O_{\amalg_\cI}} (X;\spincs),\] we first note that, since $(\Sigma,\boldalpha,\boldbeta_{\cI},P)$ is adapted to $\SI \subset \YI$ for each $\cI$, and the generators ${\bf x} \in \Torus_{\alpha} \cap \Torus_{\beta_\cI}$ lying in $O_{\SI}$ are precisely those satisfying ${\bf x} \cap P = \emptyset$, by \cite[Lem. 5.4]{MR2390347}.  Furthermore, there is a set-wise bijection \[\{{\bf x} \in \Torus_\alpha \cap \Torus_{\beta_\cI}| {\bf x} \cap P = \emptyset\} \leftrightarrow \Torus_{\alpha'} \cap \Torus_{\beta'_\cI}\] for each $\cI$, hence a bijection between the generators of $\bigoplus_{\spincs \in O_{\amalg_\cI}} (X;\spincs)$ and of $X'$.

To prove that the chain complexes are filtered quasi-isomorphic, it will suffice to show that the differentials on the complexes $X'$ and $\bigoplus_{\spincs \in O_{\amalg_\cI}} (X;\spincs)$ agree for a suitable generic almost complex structure on $\Sigma$ (inducing a generic almost complex structure on $\Sigma'$).

To this end, let \[(\cI_{i_1}, \ldots, \cI_{i_k}) \subseteq \{0,1\}^\ell\] be any ordered subset as in Section \ref{subsec:Agradings}.  
We will construct a one-to-one correspondence between embedded holomorphic surfaces representing holomorphic $k$--gons in \[W_{\alpha,\beta_{\cI_1}, \ldots, \beta_{\cI_{k}}} := \Sigma \times \PP_{k + 1}\] and \[W_{\alpha', \beta'_{\cI_1}, \ldots, \beta'_{\cI_{k}}}:= \Sigma' \times \PP_{k + 1}.\]

First, note that Corollary \ref{cor:DomainSplit} ensures the existence of a generic almost complex structure, $j_\Sigma$, on $\Sigma$ with respect to which the intersection, $\phi_{P} := \phi \cap P$, of any given candidate domain, $\phi$, for any holomorphic $k$--gon associated to any ordered subset \[(\cI_{i_1}, \ldots, \cI_{i_k}) \subseteq \{0,1\}^\ell\] as in Section \ref{subsec:Agradings} can be decomposed as \[\phi_P = \phi_A + \phi_B,\] where $\phi_A$ (resp., $\phi_B$) has boundary curves of type $A$ only (resp., of type $B$ only).

Fix such an almost complex structure, $j_\Sigma$, on $\Sigma$, and choose a family, $\{J_a|a \in \cM_{\PP_{k+1}}\}$, on $\Sigma \times \PP_{k+1}$ satisfying Lipshitz's conditions for each $k= 1, \ldots, n$.  
Now let \[(\cR, \partial \cR) \rightarrow \Sigma \times \PP_{k+1}\] be a holomorphic embedding representing a holomorphic $(k+1)$--gon, and let $\phi = \pi_\Sigma(\cR) \subset \Sigma$ be its associated domain.

By Lemma \ref{lem:DomainSplit}, $\phi_P$ splits as $\phi_A + \phi_B$, so we can lift $\phi$ uniquely to a domain, $\phi' \subset \Sigma'$, satisfying $p(\phi') = \phi$ under the covering projection, $p: \Sigma' \rightarrow \Sigma.$  We construct $\phi'$ by lifting $\phi_A$ to $\phi_A' \subset P_A$ and $\phi_B' \subset P_B$ as in Figure \ref{fig:LiftDomain} and taking \[\phi' = p^{-1}\left(\phi_{\Sigma - P}\right) + \phi_A' + \phi_B'.\] 

\begin{figure}
\begin{center}
\input{Figures/LiftDomain.pstex_t}
\end{center}
\caption{Lifting a domain, $\phi \subset \Sigma$, satisfying $\phi_P = \phi_A + \phi_B$ to a domain, $\phi' \subset \Sigma'$, by lifting $\phi_{\Sigma - P}$ to $\phi'_{\Sigma-P}$ and $\phi_A \subset P$ (resp., $\phi_B \subset P$) to $\phi_A' \subset P_A$ (resp., $\phi_B' \subset P_B$).}
\label{fig:LiftDomain}
\end{figure}
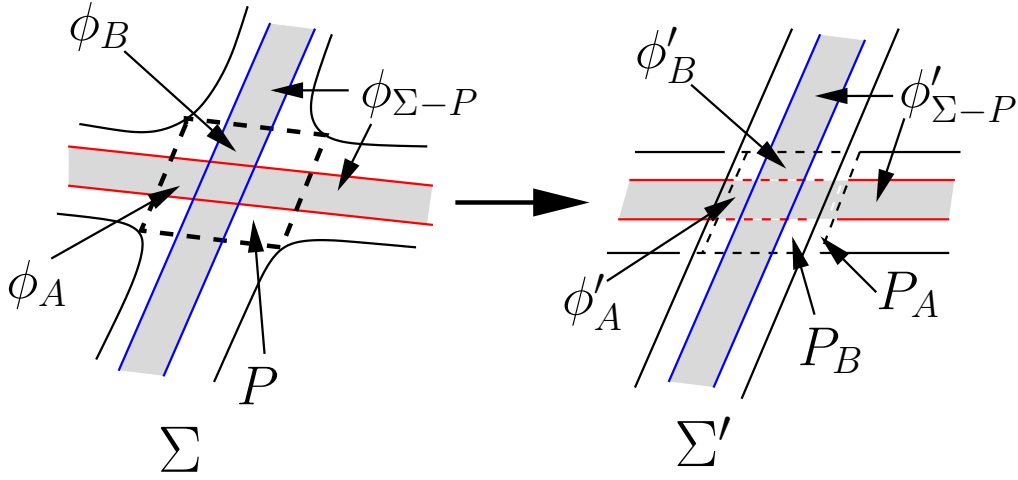

Now, pull back a family of almost complex structures, $\{J_a'\}$, on $\Sigma' \times \PP_n$ along the map \[(p \times \mathbb{I}): (\Sigma' \times \PP_{k+1}) \rightarrow (\Sigma \times \PP_{k+1}).\]  Note that $\{J_a'\}$ satisfies Lipshitz's conditions, since $\{J_a\}$ does.  We then construct an imbedded surface \[(\cR',\partial \cR') \rightarrow \Sigma' \times \PP_{k+1}\] by choosing the unique lift, \[ (\cR', \partial \cR'):=(p^{-1} \times \mathbb{I})(\cR,\partial \cR),\] satisfying the property that $\pi_{\Sigma'}(\cR') = \phi'$.  By construction, $(\cR', \partial \cR') \rightarrow \Sigma' \times \PP_{k+1}$ represents a holomorphic $(k+1)$--gon with domain, $\phi'$.

Conversely, suppose that $(\cR', \partial \cR')$ is an imbedded holomorphic surface with respect to an almost complex structure in the family, \{$J'_a\}$ constructed above.  Then \[(\cR,\partial \cR) := (p \times \mathbb{I})(\cR',\partial \cR')\] is holomorphic with respect to the the corresponding almost complex structure in the family, $\{J_a\}$.  

To see that the holomorphic map $i: \cR \rightarrow \Sigma \times \PP_{k+1}$ is also an imbedding, suppose, aiming for a contradiction, that it is not.  Then there exist points $a \neq b \in \cR$ such that $i(a) = i(b)$.  In particular, $i(a)= i(b)$ project to the same point of $\Sigma$ under the map $\pi_\Sigma$ and to the same point of $\PP_{k+1}$ under the map $\pi_{\PP}$.  

To see that this is impossible, note that since $\cR'$ is imbedded in $\Sigma' \times \PP_{k+1}$, $i(a)$ and $i(b)$ must be the images under the map $p \times \mathbb{I}$ of points $a' \in \pi_{\Sigma}^{-1}(\phi_A') \subset P_A \times \PP_{k+1}$ and $b' \in \pi^{-1}(\phi_B') \subset P_B \times \PP_{k+1}$, else the double point of $\cR$ will lift to a double point of $\cR'$.

Now let $\cR_P$ (resp., $\cR_{\Sigma - P}$) represent the preimage of $P$ (resp., $\Sigma - P$) under the map $\pi_\Sigma$ and $(\PP_{k+1})_P$ (resp., $(\PP_{k+1})_{\Sigma-P}$) its image under $\piP$.  Then \[\piP(i(a)) = \left[\piP \circ (p \times \mathbb{I})\right](a')\] must necessarily lie in a different connected component of $\PP_{k+1} - (\PP_{k+1})_{\Sigma - P}$ than \[\piP(i(b)) = \left[\piP \circ (p \times \mathbb{I})\right](b'),\] since the boundary of its connected component has curves of type $A$ only, and the boundary of the connected component containing $\piP(i(b))$ has curves of type $B$ only.

In particular, $\piP(i(a)) \neq \piP(i(b))$, and we conclude that $i: \cR \rightarrow \Sigma \times \PP_{k+1}$ is an imbedding, as desired.

\end{proof}

We are now ready to prove Theorem \ref{thm:SurfDecomp}.

\begin{proof}[Proof of Theorem \ref{thm:SurfDecomp}]
By Proposition \ref{prop:Adaptedsurf}, there exists a $(0,1)$ sutured multi-diagram, \[\left(\Sigma, \boldalpha, \boldbeta_{\{0,1\}^\ell}, P\right)\] for $L$ that is also a surface multi-diagram compatible with the collection \[\{S_{\cI} \subset Y(\cI)|\cI \in \{0,1\}^\ell\}.\]

By Proposition \ref{prop:ChainHomEquiv}, the multi-diagram, $\left(\Sigma', \boldalpha',\boldbeta'_{\{0,1\}^\ell}\right)$, obtained as in Definition \ref{defn:DecompSurf} is a $(0,1)$--sutured multi-diagram compatible with $L'$.  Furthermore, if $X^{(0,1)}$ $\left(\mbox{resp., }(X^{(0,1)})'\right)$ is the $(0,1)$--filtered chain complex corresponding to the sutured multi-diagram $\left(\Sigma, \boldalpha,\boldbeta_{\{0,1\}^\ell}\right)$ (resp., $\left(\Sigma', \boldalpha',\boldbeta'_{\{0,1\}^\ell}\right)$), then Proposition \ref{prop:ChainHomEquiv} tells us that \[\left(X^{(0,1)}\right)' \cong \bigoplus_{\spincs \in \outerspinc} (X;\spincs).\]

But \cite[Lem. 3.10]{MR2390347} tells us that, for each $\cI \in \{0,1\}^\ell$,  $\spincs \in O_{\SI}$  iff \[\langle c_1(\spincs,t),[\SI]\rangle = c(\SI,t).\] 


Hence, \[\left(X^{(0,1)}\right)' \cong \bigoplus_{\{\spincs_k(S)\,|\,k = \frac{1}{2}c(S,t)\}} \left(X^{(0,1)};\spincs_k(S)\right) \,\,\,= X_k^{(0,1)}.\]

Furthermore (see \cite[Sec. 7]{GT08082817}) any other choice of a $(0,1)$ sutured multi-diagram for $L$ (resp., $L'$) yields a filtered chain complex that is filtered quasi-isomorphic to $X^{(0,1)}$ (resp., to $\left(X^{(0,1)}\right)'$).  See Remark \ref{rmk:SpecSeqInv}).
\end{proof}


\section{Naturality of the Spectral Sequence} \label{sec:Naturality}
In this section, we discuss consequences of Theorem \ref{thm:SurfDecomp}.  We shall see that various natural geometric operations on balanced tangles can be understood in terms of surface decompositions on their sutured double branched covers, which will, in particular, imply that the algebra of the spectral sequence from Khovanov homology to sutured Floer homology, described in \cite{GT08071432} and \cite{AnnularLinks}, behaves ``as expected'' with respect to the geometric operations.

In what follows, recall that any admissible balanced tangle $T \subset D \times I$ (resp., link $\bL \subset A \times I$) can be represented by an {\em enhanced projection} (diagram) $\cP(T)$ (resp., $\cP(\bL)$).  See \cite[Sec. 5]{GT08071432} and \cite[Sec. 2]{AnnularLinks}.



\begin{theorem} \label{thm:AdjoinTrivial}(Trivial inclusion) Let $T \subset D \times I$ be an $n$--balanced tangle in the product sutured manifold $D \times I$, and let $T' \subset D \times I$ be the $(n+1)$--balanced tangle obtained from $T$ by adjoining a trivial strand separated from $T$ by a properly-imbedded $I$--invariant disk, $F$, as in Figure \ref{fig:AdjoinTrivial}.  Let $\cF(T), \cF(T')$ be the associated filtered complexes as in Definition \ref{defn:FiltCpxTangle}.  Then \[\cF(T) = \cF(T').\]
\end{theorem}

\begin{proof} 

Let $\widetilde{F}$ denote the preimage of $F$ in $\boldSigma(D \times I, T')$, which is a ($2$--component) vertical decomposing surface satisfying conditions (1) and (2) of the statement of Theorem \ref{thm:SurfDecomp}.  
  Then, as in the proof of Proposition \ref{prop:Adaptedsurf}, we can choose a Morse function on $\boldSigma(D \times I, T)$ whose gradient is everywhere tangent to $\widetilde{F}$.  In particular, (see, e.g., the constructions in \cite[Sec. 4]{MR2253454}, \cite[Sec. 2.6]{MR2113019}) any Spin$^c$ structure on $Y(\cI)$ for any $\cI \in \{0,1\}^\ell$ can be represented by a (homology class of) unit vector field that is everywhere tangent to $\widetilde{F}$.  This implies (see Definition \ref{defn:outer}) that every generator in $\cF(T')$ is outer with respect to $\widetilde{F}$.

Furthermore, decomposing along $\widetilde{F}$ produces the sutured manifold which is the disjoint union of $\boldSigma(D \times I, T)$ and a product sutured manifold.  Theorem \ref{thm:SurfDecomp} then implies \[\cF(T) = \cF(T'),\] as desired.

\end{proof}

\begin{theorem} (Stacking) \label{thm:Stack} Let $T_i \subset (D \times I)_i$, for $i=1,2$, be two $n$--balanced tangles,  and let $T_1 + T_2 \subset D \times I$ be any $n$--balanced tangle obtained by stacking a projection, $\cP(T_1)$, of $T_1$ on top of a projection, $\cP(T_2)$, of $T_2$ as in Figure \ref{fig:Stacking}.  Then \[\cF(T_1 + T_2) = \cF(T_1) \otimes \cF(T_2).\]
\end{theorem}

\begin{proof} Let $F$ denote the disk along which the two product sutured manifolds $(D \times I)_1$ and $(D \times I)_2$ are glued and $\widetilde{F}$ its preimage in $\boldSigma(D \times I, T_1 + T_2)$.  Applying finger moves as in the proof of \cite[Lem. 4.5]{MR2390347}, we move $\widetilde{F}$ to an equivalent surface satisfying conditions (1) and (2) of the statement of Theorem \ref{thm:SurfDecomp}.

As in the proof of Proposition \ref{prop:Adaptedsurf}, we choose a Morse function on $\boldSigma(D \times I, T_1 + T_2)$ whose restriction to $\widetilde{F}$ points in either the positive normal direction to $\widetilde{F}$ (along the quasi-polygon $P \subset \widetilde{F}$) or is tangent to $\widetilde{F}$ (along the vertical regions, $\widetilde{F} - \overline{P}$).  In particular, as in the proof of Theorem \ref{thm:AdjoinTrivial}, any Spin$^c$ structure on $Y(\cI)$ can be represented by a unit vector field agreeing with the above vector field on $\widetilde{F}$, hence every generator of $\cF(T_1 + T_2)$ is outer with respect to $\widetilde{F}$.

By Theorem \ref{thm:SurfDecomp}, we then obtain \[\cF(T_1 + T_2) = \cF(T_1) \otimes \cF(T_2),\] as desired.
\end{proof}

\begin{theorem} (Offset stacking) \label{thm:OffsetStack} Let $T_i$ be an $n_i$--balanced tangle in $(D \times I)_i$ for $i=1,2$.  For $k \in \Z$, let $T_1 +_k T_2$ be any $n$--balanced tangle obtained by $k$--offset stacking any projection $\cP(T_1)$ of $T_1$ atop any projection, $\cP(T_2)$, of $T_2$, as in Figure \ref{fig:OffsetStack}.  More precisely, one forms $T_1 +_k T_2$ from projections, $\cP(T_1)$ and $\cP(T_2)$, by
\begin{itemize}
  \item stacking $\cP(T_1)$ atop $\cP(T_2)$ so that the leftmost strand of $\cP(T_2)$ is $|k|$ strands to the right (resp., to the left) of the left-most strand of $\cP(T_1)$ when $k \geq 0$ (resp., $k \leq 0$),
  \item adjoining trivial strands to both $\cP(T_1)$ and $\cP(T_2)$ to ensure that both are $n$--balanced, for minimal $n$.
\end{itemize}

Then \[\cF(T_1 +_k T_2) = \cF(T_1) \otimes \cF(T_2).\]
\end{theorem}

\begin{remark} If $T_i$ is an $n_i$--balanced tangle for $i = 1, 2$, then $T_1 +_k T_2$ is an $n$--balanced tangle, with
\[n= \left\{\begin{array}{ll}
      \max(n_1, n_2 + k) & \mbox{when } k \geq 0\\
      \max(n_1 + |k|,n_2) & \mbox{when } k < 0.
            \end{array}\right.\]

With notation as above, note that $m=(n_1+n_2)-n$ is the number of overlapping strands (of the nontrivial parts) of $T_1+_k T_2$.  We will often refer to $m = (n_1+n_2) - n$ as the {\em overlap} of $T_1, T_2$ in $T_1 +_k T_2$.
\end{remark}

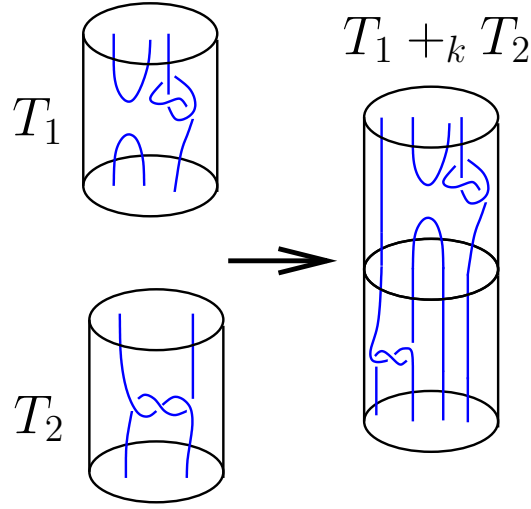
\begin{figure}
\begin{center}
\input{Figures/OffsetStack.pstex_t}
\end{center}
\caption{Offset stacking two balanced tangles $T_1$ and $T_2$ to obtain a new balanced tangle, $T_1 +_k T_2$.  In the above, $\cP(T_1)$ is stacked atop $\cP(T_2)$ with offset $k=-1$ and overlap $m=1$.}
\label{fig:OffsetStack}
\end{figure}

\begin{proof}[Proof of Theorem \ref{thm:OffsetStack}]
This is an immediate corollary of Theorems \ref{thm:AdjoinTrivial} and \ref{thm:Stack}.
\end{proof}

There is also a natural geometric relationship between links in $A \times I$ and balanced tangles in $D \times I$, along with a corresponding naturality result for the associated spectral sequences.  To understand this relationship, recall (see \cite{GT07060741}, \cite{GT08071432}) that $A \times I$ can be identified as the sutured complement of a standard unknot, $B \subset S^3$, via the identification: 
\begin{eqnarray*}
A \times I &=& \{(r,\theta, z) \,|\, r \in [1,2], \theta \in [0,2\pi), z \in [0,1]\} \subset \R^3 \cup \infty = S^3,\\
B &=& \{(r,\theta, z)\,|\,r=0\} \cup \infty \subset S^3
\end{eqnarray*}




\begin{definition} Let $a \in [0,2\pi)$.  Then \[\gamma_{a} := \{(r,\theta) \in A\,|\, \theta = a\}\] denotes the properly-imbedded arc with argument $a$, oriented outward, and \[D_{a} := \gamma_{a} \times I\] denotes the corresponding $I$--invariant disk, endowed with the product orientation.
\end{definition}

\begin{definition}  (Cutting) Let $\bL \subset A \times I$ be a link in the product sutured manifold $(A \times I, \partial A \times I)$, $\cP(\bL)$ a projection of $\bL$, and $D_a$ an $I$--invariant disk.  Then we denote by $\Psi_{a}(\cP(\bL))$ the balanced tangle projection obtained by decomposing $A \times I$ along $D_a$.
\end{definition}

\begin{definition} (Gluing) Let $T \subset D \times I$ be a balanced tangle in the product sutured manifold $(D \times I, \partial D \times I)$ and $\cP(T)$ a projection of $T$.  Then we denote by $\Psi_{a}^{-1}(\cP(T))$ the annular link projection obtained by identifying $D_+$ and $D_-$ at $\theta=a$, as in Figure \ref{fig:CutGlue}.  $\Psi_{a}^{-1}(\cP(T))$ is the braid closure of $\cP(T)$, considered as a link in $A \times I$.
\end{definition}

\begin{figure}
\begin{center}
\input{Figures/CutGlue.pstex_t}
\end{center}
\caption{(A projection of) an annular link, along with (a projection of) a balanced tangle obtained by cutting $A \times I$ along a vertical disk, $D_a = \gamma_a \times I$, for some $a \in [0,2\pi)$.}
\label{fig:CutGlue}
\end{figure}
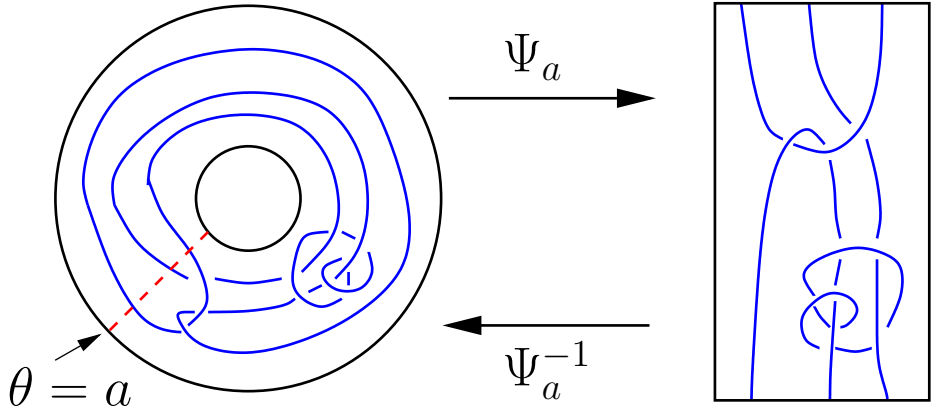

The following result was proved in \cite{GT08071432}.  We recall it here, since it provides another example of the naturality of the spectral sequence relating Khovanov homology and Heegaard Floer homology.

\begin{theorem} (Cutting) \label{thm:Cutting} \cite[Thm. 3.1]{AnnularLinks} Let $\bL \subset A \times I$ be a link, and let $T \subset D \times I$ be any balanced tangle admitting a projection, $\cP(T)$, such that $\Psi_a^{-1}(\cP(T))$ is a projection of $\bL$.  Then $\cF(T)$ is a direct summand of $\cF(\bL)$.  If there exists $a' \in [0,2\pi)$ such that \[|D_{a'} \pitchfork \bL| < |D_{a} \pitchfork \bL|,\] then $\cF(T) = 0$, i.e., the trivial direct summand of $\cF(\bL)$.
\end{theorem}

\begin{remark} Note that Theorem \ref{thm:Cutting} is proved by applying Theorem \ref{thm:SurfDecomp} to \[\widetilde{D}_a = \boldSigma(D_a, D_a \pitchfork \bL),\] a ($\Ztwo$--equivariant) Seifert surface for $\widetilde{B} \subset \boldSigma(A \times I, \bL)$, the preimage of $B$ in $\boldSigma(A \times I, \bL)$.  Furthermore, the Euler characteristic of $\widetilde{D}_a$ is given by: \[\chi(\widetilde{D}_a) = 2- |D_a \pitchfork \bL|,\] by the Riemann-Hurwitz formula.  In particular, the vanishing/non-vanishing of $\cF(T)$ detects the Thurston norm of the cohomology class Poincare dual to $[\widetilde{D}_a]$.  
\end{remark}

\subsection{Offset stacking, generalized Murasugi sum, and annular link composition}

Link projections in $A \times I$ can be composed.  This composition is closely related to both the offset stacking operation and a generalization of the Murasugi sum operation.

\begin{figure}
\begin{center}
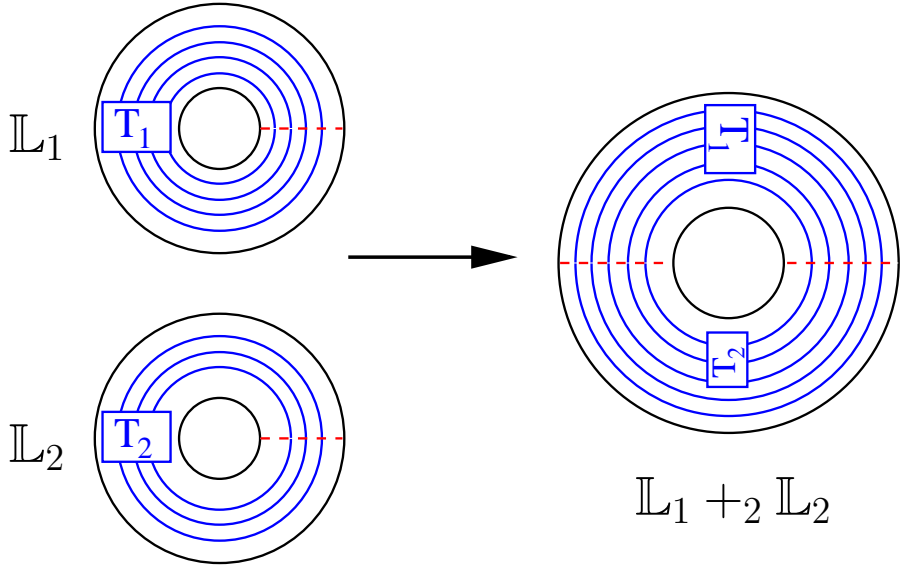
\end{center}
\caption{The annular link projection on the right is obtained by cutting open projections of $\bL_1$ and $\bL_2$, offset stacking the resulting balanced tangles, and regluing the result.  In the example above, the offset for the stacking operation is $2$ and the overlap is $2$.}
\label{fig:Composition}
\end{figure}

\begin{figure}
\begin{center}
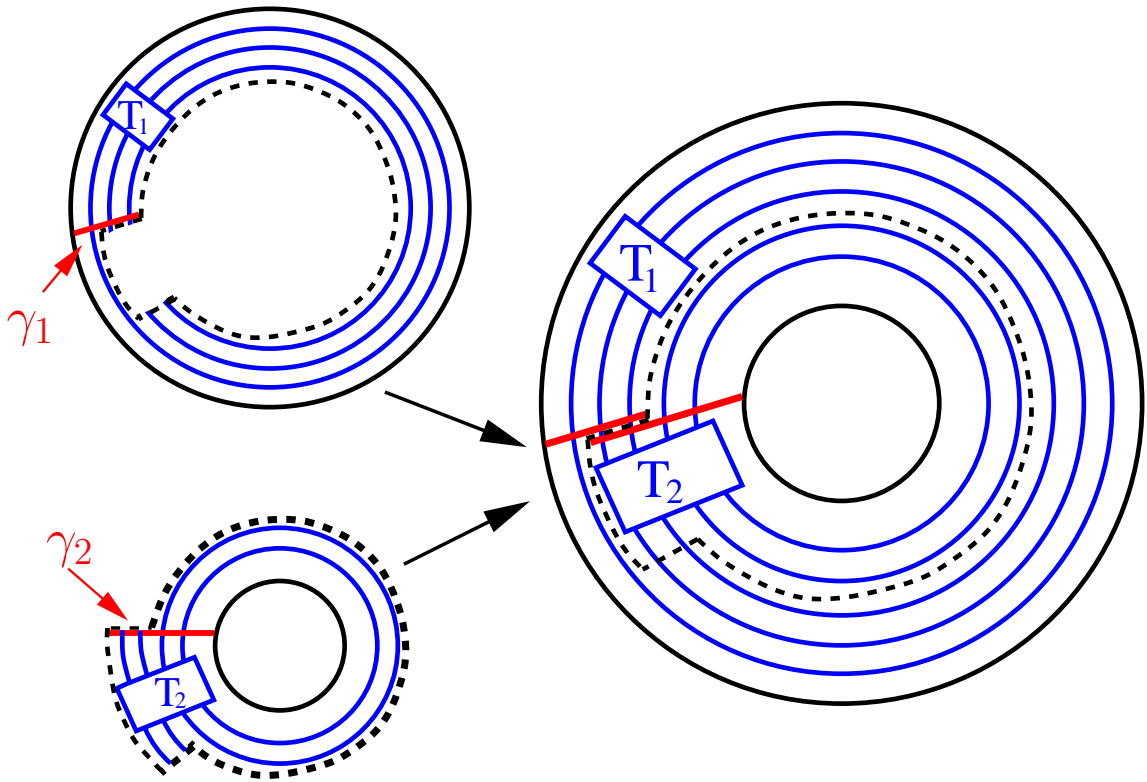
\end{center}
\caption{To obtain the figure on the right, we begin with two annular link projections, remove a neighborhood of a trivial product region in each, and identify the result.  In the double-branched cover, this corresponds to performing a generalized Murasugi sum along a subsurface of the double-branched cover of $\gamma_i \times I$ for $i=1,2$.}  
\label{fig:MurasugiSum}
\end{figure}

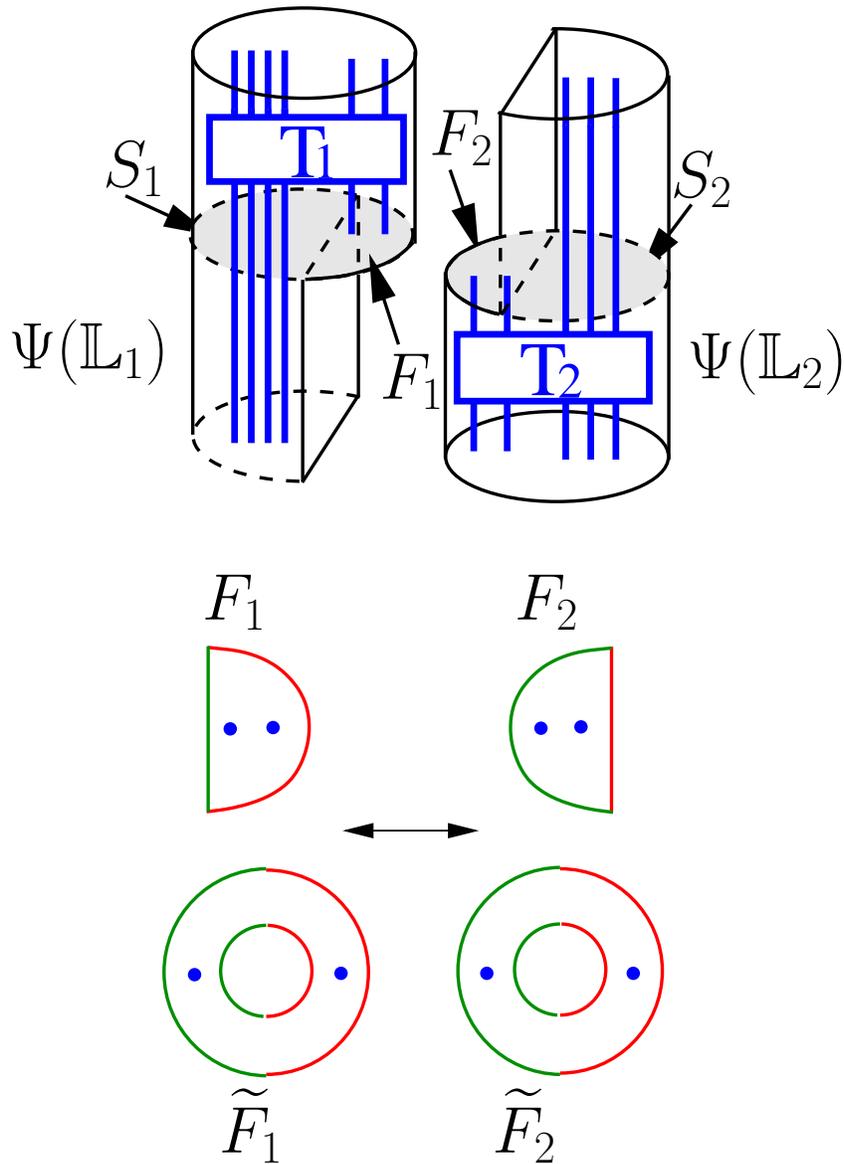
\begin{figure}
\begin{center}
\input{Figures/MurasugiStack.pstex_t}
\end{center}
\caption{Annular link composition corresponds to a generalized Murasugi sum in the double-branched cover.  In the top figure, we see two decomposed annular links being prepared to offset-stack with overlap $2$.  In the double-branched cover, fiber surfaces $\widetilde{S}_i$ (the preimages of the shaded $S_i$ for $i=1,2$) of the circle-valued Morse functions are identified along the subsurfaces $\widetilde{F}_i$.  Note that the green edges of $\partial \widetilde{F}_1$ are in the interior and the red edges are on the boundary of $\widetilde{S}_1$, whereas the green edges of $\partial \widetilde{F}_2$ are on the boundary and the red edges are in the interior of $\widetilde{S}_2$.  $\widetilde{F}_1$ and $\widetilde{F}_2$ are identified as illustrated in the bottom figure.}
\label{fig:MurasugiStack}
\end{figure}

\begin{definition} (Annular link composition) For $i=1,2$, let $\bL_i \subset (A \times I)_i$ be a link, $\cP(\bL_i)$ a projection, and $k \in \Z$.

Then the {\em $k$--offset composition}, denoted $\cP(\bL_1) +_k \cP(\bL_2)$, is defined as: 

\[\cP(\bL_1) +_k \cP(\bL_2) := \Psi_a^{-1}(\Psi_a(\cP(\bL_1))+_k \Psi_a(\cP(\bL_1))).\]

See Figure \ref{fig:Composition}.
\end{definition}

\begin{definition} (Generalized Murasugi sum)
For $i=1,2$, let $L_i$ be a nullhomologous link in the three-manifold $Y_i$ and $S_i$ a choice of Seifert surface for $L_i$.  For $i=1,2$, let $F_i \subset S_i$ be a subsurface satisfying:
\begin{itemize}
  \item Each boundary component of $F_i$ is a cyclic graph whose edges can be labeled either $1$ or $2$ such that no two adjacent edges have the same label.  In particular, there are an even number of vertices and edges.
  \item The type $1$ edges (resp., type $2$ edges) are in the boundary (resp., in the interior) of $S_1$, and the type $2$ edges (resp., type $1$ edges) are in the boundary (resp., in the interior) of $S_2$.
  \item There exists an orientation-reversing homeomorphism, $\phi: F_1 \longrightarrow F_2$, preserving the labelings of the boundary graphs.
\end{itemize}

Then one forms a new $3$--manifold, $Y_1 \#_F Y_2$ by identifying $Y_1 - (F_1 \times I)$ with $Y_2 - (F_2 \times I)$ along their common boundary, as follows.  

\begin{enumerate}
\item Note that, for $i=1,2$, \[\partial (Y_i - F_i \times I) = (F_i\times \{0\}) \cup (\partial F_i \times I) \cup (F_i \times \{1\}).\]
\item Identify $F_1 \times \{0\}$ with $F_2 \times \{0\}$ and $F_1 \times \{1\}$ with $F_2 \times \{1\}$ using $\phi$.
\item Identify $\partial F_1 \times I$ with $\partial F_2 \times I$ using the canonical (orientation-reversing) level-preserving homeomorphism.
\end{enumerate}

The {\em generalized Murasugi sum}, $S_1 *_F S_2$, is the imbedded surface in $Y_1 \#_F Y_2$ obtained by identifying $S_1 \times \{0\}$ and $S_2 \times \{0\}$ along $\phi:(F_1 \times \{0\}) \longrightarrow (F_2 \times \{0\})$.  Note that the boundary of $S_1 *_F S_2$ is a link, \[L_1 *_F L_2 = \bigcup_{i=1,2}\overline{\partial S_i - (\partial S_i \cap F_i)},\] which we also call the generalized Murasugi sum of $L_1$ and $L_2$.
\end{definition}

Note that when $F_1 \cong F_2$ is a disk, $Y_1 \#_F Y_2$ is the connected sum operation, and $S_1 *_F S_2$ is the standard Murasugi sum.




\begin{remark} Let $\bL_1 +_k \bL_2$ be any link obtained from $\bL_1, \bL_2$ by a $k$--offset composition.  Then $\boldSigma(A \times I, \bL_1 +_k \bL_2)$ is a generalized Murasugi sum of $\boldSigma(A \times I, \bL_1)$ and $\boldSigma(A \times I, \bL_2)$.  

More specifically, projection to the $\theta$ coordinate endows $A \times I$ with an $S^1$--valued Morse function with fibers $D_\theta$.  For a generic isotopy class representative of $\bL \subset A \times I$, this $S^1$--valued Morse function can be lifted to one for $\boldSigma(A \times I, \bL)$ whose fibers are $\boldSigma(D_\theta,D_\theta \cap \bL)$.  

If $\bL$ is obtained as a $k$--offset composition of $\bL_1$ and $\bL_2$ with overlap $m$, then $\boldSigma(A \times I, \bL)$ is obtained from $\boldSigma(A \times I, \bL_i)$ for $i=1,2$ by a generalized Murasugi sum along a subsurface of $\boldSigma(D_\theta,D_\theta \cap \bL)$ of genus $g$ with $b$ boundary components, where \[g = \left\lfloor\frac{m-1}{2}\right\rfloor, \mbox{ and}\]
\[b = \left\{\begin{array}{ll} 
             1 & \mbox{if $m$ is odd, and}\\
             0 & \mbox{if $m$ is even.}
             \end{array}\right.\]

See Figures \ref{fig:MurasugiSum} and \ref{fig:MurasugiStack}.
\end{remark}

\bibliography{SSSurfDecomp}
\end{document}

%% file: Figures/AdjoinTrivial.pstex_t
\begin{picture}(0,0)%
\epsfig{file=AdjoinTrivial.pstex}%
\end{picture}%
\setlength{\unitlength}{1302sp}%
\begingroup\makeatletter\ifx\SetFigFont\undefined%
\gdef\SetFigFont#1#2#3#4#5{%
  \reset@font\fontsize{#1}{#2pt}%
  \fontfamily{#3}\fontseries{#4}\fontshape{#5}%
  \selectfont}%
\fi\endgroup%
\begin{picture}(10305,5912)(301,-5568)
\put(9376,-3586){\makebox(0,0)[lb]{\smash{{\SetFigFont{14}{16.8}{\rmdefault}{\mddefault}{\updefault}{\color[rgb]{0,0,0}$T'$}%
}}}}
\put(7351,-136){\makebox(0,0)[lb]{\smash{{\SetFigFont{14}{16.8}{\rmdefault}{\mddefault}{\updefault}{\color[rgb]{0,0,0}$F$}%
}}}}
\put(301,-3511){\makebox(0,0)[lb]{\smash{{\SetFigFont{14}{16.8}{\rmdefault}{\mddefault}{\updefault}{\color[rgb]{0,0,0}$T$}%
}}}}
\end{picture}%

%% file: Figures/Stacking.pstex_t
\begin{picture}(0,0)%
\epsfig{file=Stacking.pstex}%
\end{picture}%
\setlength{\unitlength}{2013sp}%
\begingroup\makeatletter\ifx\SetFigFont\undefined%
\gdef\SetFigFont#1#2#3#4#5{%
  \reset@font\fontsize{#1}{#2pt}%
  \fontfamily{#3}\fontseries{#4}\fontshape{#5}%
  \selectfont}%
\fi\endgroup%
\begin{picture}(6992,6196)(451,-6684)
\put(451,-1861){\makebox(0,0)[lb]{\smash{{\SetFigFont{14}{16.8}{\rmdefault}{\mddefault}{\updefault}{\color[rgb]{0,0,0}$T_1$}%
}}}}
\put(6751,-5011){\makebox(0,0)[lb]{\smash{{\SetFigFont{14}{16.8}{\rmdefault}{\mddefault}{\updefault}{\color[rgb]{0,0,0}$F$}%
}}}}
\put(4876,-1261){\makebox(0,0)[lb]{\smash{{\SetFigFont{14}{16.8}{\rmdefault}{\mddefault}{\updefault}{\color[rgb]{0,0,0}$T_1+T_2$}%
}}}}
\put(451,-5611){\makebox(0,0)[lb]{\smash{{\SetFigFont{14}{16.8}{\rmdefault}{\mddefault}{\updefault}{\color[rgb]{0,0,0}$T_1$}%
}}}}
\end{picture}%

%% file: Figures/TQFT.pstex_t
\begin{picture}(0,0)%
\epsfig{file=TQFT.pstex}%
\end{picture}%
\setlength{\unitlength}{2763sp}%
\begingroup\makeatletter\ifx\SetFigFont\undefined%
\gdef\SetFigFont#1#2#3#4#5{%
  \reset@font\fontsize{#1}{#2pt}%
  \fontfamily{#3}\fontseries{#4}\fontshape{#5}%
  \selectfont}%
\fi\endgroup%
\begin{picture}(5829,2262)(1801,-2281)
\put(4201,-1786){\makebox(0,0)[lb]{\smash{{\SetFigFont{14}{16.8}{\rmdefault}{\mddefault}{\updefault}{\color[rgb]{0,0,0}Cut}%
}}}}
\put(1801,-211){\makebox(0,0)[lb]{\smash{{\SetFigFont{14}{16.8}{\rmdefault}{\mddefault}{\updefault}{\color[rgb]{0,0,0}$\bL \subset A \times I$}%
}}}}
\put(6151,-211){\makebox(0,0)[lb]{\smash{{\SetFigFont{14}{16.8}{\rmdefault}{\mddefault}{\updefault}{\color[rgb]{0,0,0}$T \subset D \times I$}%
}}}}
\end{picture}%

%% file: Figures/CappingOff.pstex_t
\begin{picture}(0,0)%
\epsfig{file=CappingOff.pstex}%
\end{picture}%
\setlength{\unitlength}{987sp}%
\begingroup\makeatletter\ifx\SetFigFont\undefined%
\gdef\SetFigFont#1#2#3#4#5{%
  \reset@font\fontsize{#1}{#2pt}%
  \fontfamily{#3}\fontseries{#4}\fontshape{#5}%
  \selectfont}%
\fi\endgroup%
\begin{picture}(23958,11025)(-3374,-8389)
\put(9526,-2911){\makebox(0,0)[lb]{\smash{{\SetFigFont{17}{20.4}{\rmdefault}{\mddefault}{\updefault}{\color[rgb]{0,0,0}$\simeq$}%
}}}}
\put(2701,-2386){\makebox(0,0)[lb]{\smash{{\SetFigFont{14}{16.8}{\rmdefault}{\mddefault}{\updefault}{\color[rgb]{0,0,0}$P_4 \times \Sigma$}%
}}}}
\put(2641,-8146){\makebox(0,0)[lb]{\smash{{\SetFigFont{14}{16.8}{\rmdefault}{\mddefault}{\updefault}{\color[rgb]{0,0,0}$e \times U_\alpha$}%
}}}}
\put(-224,-7036){\makebox(0,0)[lb]{\smash{{\SetFigFont{17}{20.4}{\rmdefault}{\mddefault}{\updefault}{\color[rgb]{0,0,1}$Y_{\alpha,\beta_{\cI^1}}$}%
}}}}
\put(7051,-7111){\makebox(0,0)[lb]{\smash{{\SetFigFont{17}{20.4}{\rmdefault}{\mddefault}{\updefault}{\color[rgb]{0,0,1}$Y_{\alpha,\beta_{\cI^k}}$}%
}}}}
\put(10576,539){\makebox(0,0)[lb]{\smash{{\SetFigFont{17}{20.4}{\rmdefault}{\mddefault}{\updefault}{\color[rgb]{0,0,1}$Y_{\alpha,\beta_{\cI^1}}$}%
}}}}
\put(19801,464){\makebox(0,0)[lb]{\smash{{\SetFigFont{17}{20.4}{\rmdefault}{\mddefault}{\updefault}{\color[rgb]{0,0,1}$Y_{\alpha,\beta_{\cI^k}}$}%
}}}}
\put(-3374,-4321){\makebox(0,0)[lb]{\smash{{\SetFigFont{14}{16.8}{\rmdefault}{\mddefault}{\updefault}{\color[rgb]{0,0,0}$e \times U_{\beta_{\cI^1}}$}%
}}}}
\put(8176,-4621){\makebox(0,0)[lb]{\smash{{\SetFigFont{14}{16.8}{\rmdefault}{\mddefault}{\updefault}{\color[rgb]{0,0,0}$e \times U_{\beta_{\cI^3}}$}%
}}}}
\put(3016,2024){\makebox(0,0)[lb]{\smash{{\SetFigFont{14}{16.8}{\rmdefault}{\mddefault}{\updefault}{\color[rgb]{0,0,0}$e \times U_{\beta_{\cI^2}}$}%
}}}}
\end{picture}%

%% file: Figures/Pentagon.pstex_t
\begin{picture}(0,0)%
\epsfig{file=Pentagon.pstex}%
\end{picture}%
\setlength{\unitlength}{1224sp}%
\begingroup\makeatletter\ifx\SetFigFont\undefined%
\gdef\SetFigFont#1#2#3#4#5{%
  \reset@font\fontsize{#1}{#2pt}%
  \fontfamily{#3}\fontseries{#4}\fontshape{#5}%
  \selectfont}%
\fi\endgroup%
\begin{picture}(26439,11656)(-503,-9714)
\put(3943,-254){\makebox(0,0)[lb]{\smash{{\SetFigFont{14}{16.8}{\rmdefault}{\mddefault}{\updefault}{\color[rgb]{0,0,0}$e_2$}%
}}}}
\put(13394,-801){\makebox(0,0)[lb]{\smash{{\SetFigFont{14}{16.8}{\rmdefault}{\mddefault}{\updefault}{\color[rgb]{0,0,0}$v_{1,2}$}%
}}}}
\put(21077,-38){\makebox(0,0)[lb]{\smash{{\SetFigFont{14}{16.8}{\rmdefault}{\mddefault}{\updefault}{\color[rgb]{0,0,0}$v_{2,3}$}%
}}}}
\put(22737,-6127){\makebox(0,0)[lb]{\smash{{\SetFigFont{14}{16.8}{\rmdefault}{\mddefault}{\updefault}{\color[rgb]{0,0,0}$v_{3,4}$}%
}}}}
\put(19084,-9446){\makebox(0,0)[lb]{\smash{{\SetFigFont{14}{16.8}{\rmdefault}{\mddefault}{\updefault}{\color[rgb]{0,0,0}$v_{4,0}$}%
}}}}
\put(15044,-9297){\makebox(0,0)[lb]{\smash{{\SetFigFont{14}{16.8}{\rmdefault}{\mddefault}{\updefault}{\color[rgb]{0,0,0}$v_{0,1}$}%
}}}}
\put(1651,-4186){\makebox(0,0)[lb]{\smash{{\SetFigFont{14}{16.8}{\rmdefault}{\mddefault}{\updefault}{\color[rgb]{0,0,0}$e_1$}%
}}}}
\put(8776,-2086){\makebox(0,0)[lb]{\smash{{\SetFigFont{14}{16.8}{\rmdefault}{\mddefault}{\updefault}{\color[rgb]{0,0,0}$e_3$}%
}}}}
\put(7876,-6736){\makebox(0,0)[lb]{\smash{{\SetFigFont{14}{16.8}{\rmdefault}{\mddefault}{\updefault}{\color[rgb]{0,0,0}$e_4$}%
}}}}
\put(3301,-8086){\makebox(0,0)[lb]{\smash{{\SetFigFont{14}{16.8}{\rmdefault}{\mddefault}{\updefault}{\color[rgb]{0,0,0}$e_0$}%
}}}}
\put(451,-2236){\makebox(0,0)[lb]{\smash{{\SetFigFont{14}{16.8}{\rmdefault}{\mddefault}{\updefault}{\color[rgb]{0,0,0}$v_{1,2}$}%
}}}}
\put(8261,-576){\makebox(0,0)[lb]{\smash{{\SetFigFont{14}{16.8}{\rmdefault}{\mddefault}{\updefault}{\color[rgb]{0,0,0}$v_{2,3}$}%
}}}}
\put(8637,-5783){\makebox(0,0)[lb]{\smash{{\SetFigFont{14}{16.8}{\rmdefault}{\mddefault}{\updefault}{\color[rgb]{0,0,0}$v_{3,4}$}%
}}}}
\put(6912,-8804){\makebox(0,0)[lb]{\smash{{\SetFigFont{14}{16.8}{\rmdefault}{\mddefault}{\updefault}{\color[rgb]{0,0,0}$v_{4,0}$}%
}}}}
\put(516,-5590){\makebox(0,0)[lb]{\smash{{\SetFigFont{14}{16.8}{\rmdefault}{\mddefault}{\updefault}{\color[rgb]{0,0,0}$v_{0,1}$}%
}}}}
\end{picture}%

%% file: Figures/Polygon.pstex_t
\begin{picture}(0,0)%
\epsfig{file=Polygon.pstex}%
\end{picture}%
\setlength{\unitlength}{1973sp}%
\begingroup\makeatletter\ifx\SetFigFont\undefined%
\gdef\SetFigFont#1#2#3#4#5{%
  \reset@font\fontsize{#1}{#2pt}%
  \fontfamily{#3}\fontseries{#4}\fontshape{#5}%
  \selectfont}%
\fi\endgroup%
\begin{picture}(15475,9622)(-2249,-8398)
\put(-2249,-6436){\makebox(0,0)[lb]{\smash{{\SetFigFont{25}{30.0}{\rmdefault}{\mddefault}{\updefault}{\color[rgb]{0,0,0}$N_{\pm}(\gamma_i)$}%
}}}}
\put(5360,-4255){\makebox(0,0)[lb]{\smash{{\SetFigFont{34}{40.8}{\rmdefault}{\mddefault}{\updefault}{\color[rgb]{0,0,0}$\sim$}%
}}}}
\put(1426,-3811){\makebox(0,0)[lb]{\smash{{\SetFigFont{25}{30.0}{\rmdefault}{\mddefault}{\updefault}{\color[rgb]{0,0,0}$\gamma_i$}%
}}}}
\put(8101,-4561){\makebox(0,0)[lb]{\smash{{\SetFigFont{25}{30.0}{\rmdefault}{\mddefault}{\updefault}{\color[rgb]{0,0,0}$\gamma_i$}%
}}}}
\end{picture}%

%% file: Figures/Reebchords.pstex_t
\begin{picture}(0,0)%
\epsfig{file=Reebchords.pstex}%
\end{picture}%
\setlength{\unitlength}{1500sp}%
\begingroup\makeatletter\ifx\SetFigFont\undefined%
\gdef\SetFigFont#1#2#3#4#5{%
  \reset@font\fontsize{#1}{#2pt}%
  \fontfamily{#3}\fontseries{#4}\fontshape{#5}%
  \selectfont}%
\fi\endgroup%
\begin{picture}(16474,12823)(-2228,-11501)
\put(-2228,-8689){\makebox(0,0)[lb]{\smash{{\SetFigFont{29}{34.8}{\rmdefault}{\mddefault}{\updefault}{\color[rgb]{0,0,0}$\widehat{P}$}%
}}}}
\put(3957,-5426){\makebox(0,0)[lb]{\smash{{\SetFigFont{14}{16.8}{\rmdefault}{\mddefault}{\updefault}{\color[rgb]{1,0,0}type A}%
}}}}
\put(3979,-5970){\makebox(0,0)[lb]{\smash{{\SetFigFont{14}{16.8}{\rmdefault}{\mddefault}{\updefault}{\color[rgb]{1,0,0}curves}%
}}}}
\put(4096,-9295){\makebox(0,0)[lb]{\smash{{\SetFigFont{14}{16.8}{\rmdefault}{\mddefault}{\updefault}{\color[rgb]{0,0,1}type B}%
}}}}
\put(4096,-9839){\makebox(0,0)[lb]{\smash{{\SetFigFont{14}{16.8}{\rmdefault}{\mddefault}{\updefault}{\color[rgb]{0,0,1}curves}%
}}}}
\put(8142,-950){\makebox(0,0)[lb]{\smash{{\SetFigFont{29}{34.8}{\rmdefault}{\mddefault}{\updefault}{\color[rgb]{0,0,0}*}%
}}}}
\put(7728,-5167){\makebox(0,0)[lb]{\smash{{\SetFigFont{29}{34.8}{\rmdefault}{\mddefault}{\updefault}{\color[rgb]{0,0,0}*}%
}}}}
\put(7750,-7537){\makebox(0,0)[lb]{\smash{{\SetFigFont{29}{34.8}{\rmdefault}{\mddefault}{\updefault}{\color[rgb]{0,0,0}*}%
}}}}
\put(7554,-2385){\makebox(0,0)[lb]{\smash{{\SetFigFont{29}{34.8}{\rmdefault}{\mddefault}{\updefault}{\color[rgb]{0,0,0}*}%
}}}}
\put(7728,-3841){\makebox(0,0)[lb]{\smash{{\SetFigFont{29}{34.8}{\rmdefault}{\mddefault}{\updefault}{\color[rgb]{0,0,0}*}%
}}}}
\put(8815,-2777){\makebox(0,0)[lb]{\smash{{\SetFigFont{29}{34.8}{\rmdefault}{\mddefault}{\updefault}{\color[rgb]{0,0,0}*}%
}}}}
\put(7728,-8973){\makebox(0,0)[lb]{\smash{{\SetFigFont{29}{34.8}{\rmdefault}{\mddefault}{\updefault}{\color[rgb]{0,0,0}*}%
}}}}
\put(7575,-11232){\makebox(0,0)[lb]{\smash{{\SetFigFont{29}{34.8}{\rmdefault}{\mddefault}{\updefault}{\color[rgb]{0,0,0}*}%
}}}}
\put(11011,-2450){\makebox(0,0)[lb]{\smash{{\SetFigFont{29}{34.8}{\rmdefault}{\mddefault}{\updefault}{\color[rgb]{0,0,0}$Z_1$}%
}}}}
\put(11076,-6232){\makebox(0,0)[lb]{\smash{{\SetFigFont{29}{34.8}{\rmdefault}{\mddefault}{\updefault}{\color[rgb]{0,0,0}$Z_2$}%
}}}}
\put(11077,-10167){\makebox(0,0)[lb]{\smash{{\SetFigFont{29}{34.8}{\rmdefault}{\mddefault}{\updefault}{\color[rgb]{0,0,0}$Z_3$}%
}}}}
\put(7380,878){\makebox(0,0)[lb]{\smash{{\SetFigFont{14}{16.8}{\rmdefault}{\mddefault}{\updefault}{\color[rgb]{0,0,0}basepoints}%
}}}}
\end{picture}%

%% file: Figures/LiftDomain.pstex_t
\begin{picture}(0,0)%
\epsfig{file=LiftDomain.pstex}%
\end{picture}%
\setlength{\unitlength}{1855sp}%
\begingroup\makeatletter\ifx\SetFigFont\undefined%
\gdef\SetFigFont#1#2#3#4#5{%
  \reset@font\fontsize{#1}{#2pt}%
  \fontfamily{#3}\fontseries{#4}\fontshape{#5}%
  \selectfont}%
\fi\endgroup%
\begin{picture}(13563,6576)(-374,-7147)
\put(1651,-6886){\makebox(0,0)[lb]{\smash{{\SetFigFont{25}{30.0}{\rmdefault}{\mddefault}{\updefault}{\color[rgb]{0,0,0}$\Sigma$}%
}}}}
\put(8551,-6811){\makebox(0,0)[lb]{\smash{{\SetFigFont{25}{30.0}{\rmdefault}{\mddefault}{\updefault}{\color[rgb]{0,0,0}$\Sigma'$}%
}}}}
\put(11551,-2011){\makebox(0,0)[lb]{\smash{{\SetFigFont{20}{24.0}{\rmdefault}{\mddefault}{\updefault}{\color[rgb]{0,0,0}$\phi'_{\Sigma - P}$}%
}}}}
\put(451,-1111){\makebox(0,0)[lb]{\smash{{\SetFigFont{20}{24.0}{\rmdefault}{\mddefault}{\updefault}{\color[rgb]{0,0,0}$\phi_B$}%
}}}}
\put(2701,-5986){\makebox(0,0)[lb]{\smash{{\SetFigFont{20}{24.0}{\rmdefault}{\mddefault}{\updefault}{\color[rgb]{0,0,0}$P$}%
}}}}
\put(8026,-1486){\makebox(0,0)[lb]{\smash{{\SetFigFont{20}{24.0}{\rmdefault}{\mddefault}{\updefault}{\color[rgb]{0,0,0}$\phi_B'$}%
}}}}
\put(-374,-4561){\makebox(0,0)[lb]{\smash{{\SetFigFont{20}{24.0}{\rmdefault}{\mddefault}{\updefault}{\color[rgb]{0,0,0}$\phi_A$}%
}}}}
\put(4351,-2011){\makebox(0,0)[lb]{\smash{{\SetFigFont{20}{24.0}{\rmdefault}{\mddefault}{\updefault}{\color[rgb]{0,0,0}$\phi_{\Sigma - P}$}%
}}}}
\put(11251,-4711){\makebox(0,0)[lb]{\smash{{\SetFigFont{20}{24.0}{\rmdefault}{\mddefault}{\updefault}{\color[rgb]{0,0,0}$P_A$}%
}}}}
\put(7051,-4786){\makebox(0,0)[lb]{\smash{{\SetFigFont{20}{24.0}{\rmdefault}{\mddefault}{\updefault}{\color[rgb]{0,0,0}$\phi_A'$}%
}}}}
\put(10201,-5461){\makebox(0,0)[lb]{\smash{{\SetFigFont{20}{24.0}{\rmdefault}{\mddefault}{\updefault}{\color[rgb]{0,0,0}$P_B$}%
}}}}
\end{picture}%

%% file: Figures/OffsetStack.pstex_t
\begin{picture}(0,0)%
\epsfig{file=OffsetStack.pstex}%
\end{picture}%
\setlength{\unitlength}{2013sp}%
\begingroup\makeatletter\ifx\SetFigFont\undefined%
\gdef\SetFigFont#1#2#3#4#5{%
  \reset@font\fontsize{#1}{#2pt}%
  \fontfamily{#3}\fontseries{#4}\fontshape{#5}%
  \selectfont}%
\fi\endgroup%
\begin{picture}(8471,6196)(249,-6684)
\put(4321,-1140){\makebox(0,0)[lb]{\smash{{\SetFigFont{20}{24.0}{\rmdefault}{\mddefault}{\updefault}{\color[rgb]{0,0,0}$T_1 +_k T_2$}%
}}}}
\put(249,-2141){\makebox(0,0)[lb]{\smash{{\SetFigFont{20}{24.0}{\rmdefault}{\mddefault}{\updefault}{\color[rgb]{0,0,0}$T_1$}%
}}}}
\put(267,-5810){\makebox(0,0)[lb]{\smash{{\SetFigFont{20}{24.0}{\rmdefault}{\mddefault}{\updefault}{\color[rgb]{0,0,0}$T_2$}%
}}}}
\end{picture}%

%% file: Figures/CutGlue.pstex_t
\begin{picture}(0,0)%
\epsfig{file=CutGlue.pstex}%
\end{picture}%
\setlength{\unitlength}{2210sp}%
\begingroup\makeatletter\ifx\SetFigFont\undefined%
\gdef\SetFigFont#1#2#3#4#5{%
  \reset@font\fontsize{#1}{#2pt}%
  \fontfamily{#3}\fontseries{#4}\fontshape{#5}%
  \selectfont}%
\fi\endgroup%
\begin{picture}(10371,4702)(1,-6064)
\put(5551,-5686){\makebox(0,0)[lb]{\smash{{\SetFigFont{20}{24.0}{\rmdefault}{\mddefault}{\updefault}{\color[rgb]{0,0,0}$\Psi_a^{-1}$}%
}}}}
\put(  1,-5911){\makebox(0,0)[lb]{\smash{{\SetFigFont{20}{24.0}{\rmdefault}{\mddefault}{\updefault}{\color[rgb]{0,0,0}$\theta = a$}%
}}}}
\put(5551,-2161){\makebox(0,0)[lb]{\smash{{\SetFigFont{20}{24.0}{\rmdefault}{\mddefault}{\updefault}{\color[rgb]{0,0,0}$\Psi_a$}%
}}}}
\end{picture}%

%% file: Figures/Composition.pstex_t
\begin{picture}(0,0)%
\epsfig{file=Composition.pstex}%
\end{picture}%
\setlength{\unitlength}{1934sp}%
\begingroup\makeatletter\ifx\SetFigFont\undefined%
\gdef\SetFigFont#1#2#3#4#5{%
  \reset@font\fontsize{#1}{#2pt}%
  \fontfamily{#3}\fontseries{#4}\fontshape{#5}%
  \selectfont}%
\fi\endgroup%
\begin{picture}(11433,7219)(226,-7758)
\put(226,-2461){\makebox(0,0)[lb]{\smash{{\SetFigFont{20}{24.0}{\rmdefault}{\mddefault}{\updefault}{\color[rgb]{0,0,0}$\bL_1$}%
}}}}
\put(226,-6436){\makebox(0,0)[lb]{\smash{{\SetFigFont{20}{24.0}{\rmdefault}{\mddefault}{\updefault}{\color[rgb]{0,0,0}$\bL_2$}%
}}}}
\put(8251,-7111){\makebox(0,0)[lb]{\smash{{\SetFigFont{20}{24.0}{\rmdefault}{\mddefault}{\updefault}{\color[rgb]{0,0,0}$\bL_1 +_2 \bL_2$}%
}}}}
\end{picture}%

%% file: Figures/MurasugiSum.pstex_t
\begin{picture}(0,0)%
\epsfig{file=MurasugiSum.pstex}%
\end{picture}%
\setlength{\unitlength}{1934sp}%
\begingroup\makeatletter\ifx\SetFigFont\undefined%
\gdef\SetFigFont#1#2#3#4#5{%
  \reset@font\fontsize{#1}{#2pt}%
  \fontfamily{#3}\fontseries{#4}\fontshape{#5}%
  \selectfont}%
\fi\endgroup%
\begin{picture}(14595,9947)(-3707,-8306)
\put(-3200,-5415){\makebox(0,0)[lb]{\smash{{\SetFigFont{20}{24.0}{\rmdefault}{\mddefault}{\updefault}{\color[rgb]{1,0,0}$\gamma_2$}%
}}}}
\put(-3707,-2578){\makebox(0,0)[lb]{\smash{{\SetFigFont{20}{24.0}{\rmdefault}{\mddefault}{\updefault}{\color[rgb]{1,0,0}$\gamma_1$}%
}}}}
\end{picture}%

%% file: Figures/MurasugiStack.pstex_t
\begin{picture}(0,0)%
\epsfig{file=MurasugiStack.pstex}%
\end{picture}%
\setlength{\unitlength}{2763sp}%
\begingroup\makeatletter\ifx\SetFigFont\undefined%
\gdef\SetFigFont#1#2#3#4#5{%
  \reset@font\fontsize{#1}{#2pt}%
  \fontfamily{#3}\fontseries{#4}\fontshape{#5}%
  \selectfont}%
\fi\endgroup%
\begin{picture}(7310,10448)(-4947,-14011)
\put(-406,-9077){\makebox(0,0)[lb]{\smash{{\SetFigFont{25}{30.0}{\rmdefault}{\mddefault}{\updefault}{\color[rgb]{0,0,0}$F_2$}%
}}}}
\put(-3208,-9078){\makebox(0,0)[lb]{\smash{{\SetFigFont{25}{30.0}{\rmdefault}{\mddefault}{\updefault}{\color[rgb]{0,0,0}$F_1$}%
}}}}
\put(-3059,-13863){\makebox(0,0)[lb]{\smash{{\SetFigFont{25}{30.0}{\rmdefault}{\mddefault}{\updefault}{\color[rgb]{0,0,0}$\widetilde{F}_1$}%
}}}}
\put(-605,-13863){\makebox(0,0)[lb]{\smash{{\SetFigFont{25}{30.0}{\rmdefault}{\mddefault}{\updefault}{\color[rgb]{0,0,0}$\widetilde{F}_2$}%
}}}}
\put(-1615,-7096){\makebox(0,0)[lb]{\smash{{\SetFigFont{25}{30.0}{\rmdefault}{\mddefault}{\updefault}{\color[rgb]{0,0,0}$F_1$}%
}}}}
\put(-1173,-4907){\makebox(0,0)[lb]{\smash{{\SetFigFont{25}{30.0}{\rmdefault}{\mddefault}{\updefault}{\color[rgb]{0,0,0}$F_2$}%
}}}}
\put(-4100,-5190){\makebox(0,0)[lb]{\smash{{\SetFigFont{25}{30.0}{\rmdefault}{\mddefault}{\updefault}{\color[rgb]{0,0,0}$S_1$}%
}}}}
\put(994,-5284){\makebox(0,0)[lb]{\smash{{\SetFigFont{25}{30.0}{\rmdefault}{\mddefault}{\updefault}{\color[rgb]{0,0,0}$S_2$}%
}}}}
\put(1142,-6913){\makebox(0,0)[lb]{\smash{{\SetFigFont{25}{30.0}{\rmdefault}{\mddefault}{\updefault}{\color[rgb]{0,0,0}$\Psi(\bL_2)$}%
}}}}
\put(-4947,-6825){\makebox(0,0)[lb]{\smash{{\SetFigFont{25}{30.0}{\rmdefault}{\mddefault}{\updefault}{\color[rgb]{0,0,0}$\Psi(\bL_1)$}%
}}}}
\end{picture}%